\documentclass[preprint,12pt]{elsarticle}
\usepackage{lmodern}%
\usepackage{url}
\usepackage{adjustbox,array,tikz,algorithm,algpseudocode}\usetikzlibrary{positioning}
\usepackage{amsmath,amssymb,amsthm,booktabs,graphicx,xcolor,hyperref}\usepackage[expansion=false]{microtype}
\hypersetup{hidelinks}%
\newtheorem{lemma}{Lemma}\newtheorem{proposition}{Proposition}
\makeatletter\@ifclassloaded{elsarticle}{\newcommand{\appref}[1]{\ref{#1}}}{\newcommand{\appref}[1]{Appendix~\ref{#1}}}\makeatother

\makeatletter\renewcommand\paragraph{\@startsection{paragraph}{4}{\z@}{3.25ex \@plus1ex \@minus.2ex}{-1em}{\normalfont\normalsize\bfseries}}\makeatother

\journal{Computers \& Mathematics with Applications}
\begin{document}
\begin{frontmatter}
\title{Learned smoothing in operator-adapted multigrid: controlled accuracy gains and computational limits}
\author[aff1]{Kumar Kaushik Ranjan\corref{cor1}}
\ead{kumarr@tmi.tolani.edu}
\cortext[cor1]{Corresponding author}
\affiliation[aff1]{organization={Department of Mechanical Engineering, Tolani Maritime Institute},
            addressline={Induri, Talegaon Dabhade},
            city={Pune},
            postcode={410507},
            state={Maharashtra},
            country={India}}
\begin{abstract}
Learned multigrid components are often evaluated at a fixed iteration budget, but such comparisons do not distinguish local
approximation gains from practical solver efficiency. We study this distinction for high-contrast elliptic problems on
unstructured finite element meshes. A learned V-cycle uses operator-dependent transfers, Galerkin coarse operators, near-kernel
deflation and an energy-norm safeguard; only its smoothers are trained. The architecture represents a specified class of fixed
polynomial smoothers and is positively homogeneous in the load. On fresh conductive-channel problems, one learned pass reduces the
energy error by factors of 5.7--6.1 relative to the same cycle with four-step Jacobi-preconditioned conjugate-gradient smoothing;
the corresponding factors are 1.38--1.45 for insulating barriers. In a separate fixed-input test on channel problems, the
fine-level learned smoother attains a smaller energy-norm error than the exact minimum over a five-dimensional Krylov space of
its own input on 97--100 of 100 fresh inputs for each position and training initialisation; this does not establish superiority
over all polynomial-smoothed cycles. When the classical smoothing budget is increased, the paired geometric-mean accuracy crosses
that of the learned cycle between 8 and 14 steps per smoothing call on channels and between 5 and 8 on barriers. The measured CPU
implementation offers no computational advantage over the tested classical solvers, and load changes and floating islands further
restrict its accuracy.

\end{abstract}
\begin{keyword} learned multigrid \sep smoothers \sep Krylov methods \sep operator-adapted wavelets \sep high-contrast elliptic problems \sep finite elements \end{keyword}
\end{frontmatter}
\section{Introduction}\label{sec:intro}
Parametric studies of heterogeneous media solve the same elliptic problem for many coefficient fields. At high contrast the
natural measure of algebraic accuracy is the energy norm, and classical solvers reach it with operator-dependent multilevel
structure: algebraic multigrid builds its interpolation from the matrix \cite{vanek1996,xu2017acta}; operator-adapted wavelets
and gamblets decompose the solution space into energy-orthogonal levels \cite{amaratunga2006,sudarshan2006,owhadi2017};
localized orthogonal decomposition and spectral coarse spaces make the coarse level robust to the contrast
\cite{malqvist2014,peterseim2016,hellman2017,galvis2010,spillane2014geneo}. A natural way to bring learning into this
machinery is to keep the structure and learn one component. Learned smoothers \cite{huang2023}, learned transfer operators
\cite{greenfeld2019,luz2020} and graph-network multilevel preconditioners \cite{wangluo2026,zhanglisaad2026} follow this route,
and they are usually assessed as iterative solvers or preconditioners.

We ask a narrower question about \emph{one pass} through a fixed multilevel cycle on a hierarchy adapted to the operator of each
sample: what does learned smoothing add? A reported accuracy gain can answer three different questions. (i)~Does the learned
smoother improve on one fixed classical smoother configuration with the same number of stages? (ii)~Does it improve on the best
that a specified classical local space can deliver for the \emph{same input}? (iii)~Does the gain survive when the classical
cycle is given more work, and does it lower the complete cost of reaching a given accuracy? The answers need not agree, and a
claim resting on (i) alone can be undone by a stronger classical control. Weak baselines and reporting biases in learned PDE
solvers are documented \cite{mcgreivy2024}, and cost-aware measures such as breakeven complexity have been proposed
\cite{zhang2026breakeven}; the contribution here is a controlled demonstration, in one explicit multilevel construction, of how the three questions separate.

We construct OA-VCB, a learned V-cycle on an unstructured finite element hierarchy (Section~\ref{sec:method}). Its transfers
are computed from the operator of each sample by a few Jacobi steps towards the Schur-complement prolongation of
operator-adapted wavelets; its coarse operators are Galerkin products; its coarsest level is solved exactly; floating
conductive regions, whose indicator functions are near-kernel vectors, are detected from the matrix and deflated; and the output
is the energy-optimal combination of the cycle's directions. Only the smoothers are learned. Section~\ref{sec:controls} defines
the controls: the same cycle with classical smoothers (Chebyshev, and Jacobi- or symmetric Gauss--Seidel-preconditioned
conjugate-gradient steps inside each smoothing call), the exact energy minimum over a Krylov space for a fixed smoother input,
and independently implemented whole-system solvers.

The answers separate. On conductive-channel problems the learned pass is $5.7$--$6.1$ times more accurate than a four-step
Jacobi-PCG smoother in the same cycle, and for the input of a single smoothing call it attains a smaller energy-norm error than
the exact minimum over the five-dimensional Krylov space of that input on $97$--$100$ of $100$ fresh inputs
(Sections~\ref{sec:full} and~\ref{sec:local}). On the reserve operators, however, the paired geometric-mean accuracy of the
classical cycle crosses that of OA-VCB between the tested budgets of eight and fourteen Jacobi-PCG steps per smoothing call on
channels, and between five and eight on barriers; fourteen steps are more accurate on both families (Section~\ref{sec:budget}).
The measured implementation is slower than classical whole-system solvers of the same accuracy by factors of about $12$ to $98$
on the two measured platforms (Section~\ref{sec:cost}), and changes of the load and floating islands restrict the gain further
(Section~\ref{sec:limits}).

Our contributions are the following.
\begin{itemize}
\item An explicit learned-cycle construction with a proved representation property for fixed polynomial smoothers and qualified
properties of positive homogeneity and energy safeguard, which allow component-controlled comparisons. The hierarchy, the
deflation and the Ritz step are classical in origin; we do not present them as new.
\item A numerical distinction between local fixed-input Krylov-space performance and full-cycle performance as the classical
smoothing budget grows, supported by fresh confirmation data and three separately trained initialisations.
\item An evaluation of cost, load dependence and barrier failure cases that bounds the usefulness of the observed gains, with
the controls and per-operator records needed by later studies.
\end{itemize}
We make no claim of a practical solver. Experiments are two-dimensional with $P_1$ elements, two coefficient families on one
family of perforated domains, and the architecture as defined here; further coefficient and domain variations are reported as
supplementary scope checks. The structural properties are proved, the accuracy results are empirical, and we give no
convergence or approximation theory for the learned smoothers. Section~\ref{sec:related} places the work in the literature, and
Section~\ref{sec:discussion} draws the implications for evaluating learned multilevel components.

\section{Related work}\label{sec:related}

\paragraph{Learned multigrid.} Learned prolongations for geometric multigrid \cite{greenfeld2019}, learned restriction and
prolongation matrices \cite{katrutsa2017}, learned algebraic-multigrid prolongations predicted by graph networks \cite{luz2020}
and learned smoothers \cite{huang2023,fanaskov2021} are trained to improve the error propagation of a cycle that is then
iterated. Recent work learns smoothing and transfer operators with graph convolutions \cite{wangluo2026}, learns an
algebraic-multigrid-type hierarchy for use as a preconditioner \cite{zhanglisaad2026}, tunes smoother parameters inside an
algebraic multigrid pressure solver \cite{chillon2026}, or uses a neural operator as preconditioner or coarse correction inside a
Krylov method \cite{rudikov2024,kopanicakova2025,melchers2026,liu2025nodal}. Learned smoothers are therefore established prior
art. Our setting differs in what is measured: one pass through a fixed architecture whose transfers are computed, not learned,
so that the same pass with classical smoothers is a controlled comparison. Multigrid structure has also entered operator learning
as a parameterisation \cite{hexu2019mgnet,he2024mgno,hu2025hmgno}; there the transfers are learned or fixed by the grid.

\paragraph{Operator-adapted multiresolution and high-contrast coarse spaces.} The transfers come from operator-adapted finite
element wavelets
\cite{sudarshan2003,sudarshan2005proc,amaratunga2006,sudarshan2006,sudarshan2005thesis,dheedene2005thesis} and gamblets
\cite{owhadi2017,owhadi2019book}, whose multilevel decomposition is the Schur-complement recursion of
Lemma~\ref{lem:routev}; the same objects are ideal interpolation in algebraic multigrid \cite{vanek1996,xu2017acta}, and the
construction has been extended to differential forms, heterogeneous elasticity and multiscale electromagnetics
\cite{budninskiy2019,chen2019tog,chen2021tog,sik2025}. For high contrast, localisation depends on the contrast unless the coarse
space resolves the conductive structures, and contrast-robust coarse spaces add spectral or problem-adapted functions
\cite{malqvist2014,kornhuber2018,peterseim2016,hellman2017,galvis2010,spillane2014geneo,babuska2011,chung2018cem,ma2022}.
Deflation \cite{frank2001,tang2009} targets the low-energy components represented by selected vectors; the indicator deflation of
Section~\ref{sec:deflation} is its simplest instance. Neural networks have been used to approximate such constructions
\cite{kropfl2021,kropfl2022,kropfl2025,rudikov2025lsino,qiu2025}; here the hierarchy is computed from each sample and learning
acts on the residuals that remain.

\paragraph{Classical smoothers used as controls.} Polynomial smoothers are competitive with Gauss--Seidel in parallel multigrid
\cite{adams2003}, and optimal polynomial smoothers for V-cycles have been analysed \cite{lottes2022}; Krylov methods inside the
cycle give the K-cycle \cite{notay2008}. These are the classical alternatives against which a learned smoother must be measured,
which is why our controls use Chebyshev and Krylov smoothing calls and a local Krylov-space minimum
(Section~\ref{sec:controls}).

\paragraph{Evaluation of learned PDE solvers.} Reporting biases and weak baselines in learned fluid-related PDE solvers are
documented in \cite{mcgreivy2024}; a differentiable finite-difference model with coefficient-valued weights outperforms Fourier
neural operators, U-Nets and ResNets on scalar time-dependent PDEs in \cite{chatain2025}; and breakeven complexity counts the
forward solves needed before a learned solver is cost-effective \cite{zhang2026breakeven}. We do not claim these principles as
new. Our contribution is the controlled separation, in a hybrid multilevel setting, of a local numerical optimum, full-cycle
accuracy at larger classical budgets and complete cost.

\section{An operator-adapted learned V-cycle}\label{sec:method}

\subsection{Setting}\label{sec:setting}
We consider
\begin{equation}\label{eq:pde}
  -\nabla\cdot\big(K(x)\nabla u\big) = f \quad\text{in } \Omega, \qquad u = 0 \text{ on } \Gamma_D,
  \qquad K\nabla u\cdot n = 0 \text{ on } \Gamma_N,
\end{equation}
with a symmetric, uniformly positive definite coefficient $K$, piecewise constant on the elements and varying by several orders
of magnitude. The main experiments use isotropic $K = a\,I$ (conductive channels or insulating barriers); a tensor-valued case
and a plane-strain elasticity case are supplementary scope checks (Section~\ref{sec:scope}). From an unstructured coarse
triangulation $\mathcal{T}_0$, $L$ uniform refinements give nested $P_1$ spaces $V_0\subset\cdots\subset V_L$. On level $k$ the
nodes of $V_{k-1}$ (``coarse'', $C$) precede the nodes created by the refinement (``fine'', $F$), and the geometric prolongation
is $P_k = [I;\,P_{k,F}]$ with $P_{k,F}$ nodal interpolation at edge midpoints. On the finest level the discrete problem is
$Au = F$ with $A = A_L$ symmetric positive definite. A Galerkin solve on the level-$k_0$ space gives $u_c$; with $r = F - Au_c$
and $e = u - u_c = A^{-1}r$, the learned map returns a correction $c = \mathcal{N}(r; A)\approx e$ in one pass. The case $k_0 = 0$
supplies no external coarse solution ($u_c = 0$); the network still contains its own exact coarsest solve. We measure
$\|v\|_A = \sqrt{v^\top A v}$ and the relative energy error $e(\hat u) = \|u - \hat u\|_A/\|u\|_A$.

\subsection{Transfers computed from the operator}\label{sec:routev}
In $C/F$ ordering, write the level-$k$ operator with blocks $A_{CC}$, $A_{CF}$, $A_{FC}$, $A_{FF}$, and define
\begin{equation}\label{eq:ideal}
  \Phi_k^\star = \begin{bmatrix} I\\ -A_{FF}^{-1}A_{FC}\end{bmatrix}, \qquad
  S_k = (\Phi_k^\star)^{\!\top} A_k \Phi_k^\star = A_{CC}-A_{CF}A_{FF}^{-1}A_{FC}.
\end{equation}
\begin{lemma}[two-level block diagonalisation]\label{lem:routev}
With $E_F = [0;\,I]$, $E_F^\top A_k\Phi_k^\star = 0$ and
\begin{equation}\label{eq:inverse}
  A_k^{-1} = \Phi_k^\star S_k^{-1} (\Phi_k^\star)^{\!\top} + E_F A_{FF}^{-1} E_F^\top .
\end{equation}
\end{lemma}
Applied recursively with $A_{k-1} = S_k$, \eqref{eq:inverse} is the exact multilevel decomposition behind operator-adapted
wavelets and, with nodal measurement functions, gamblets; the columns of $\Phi_k^\star$ are $A$-harmonic extensions of the coarse
nodal values, and it is the ideal interpolation of algebraic multigrid (proof in \appref{app:proofs}). The matrix $\Phi_k^\star$ is
dense in general, and its decay depends on the contrast, so the computational transfer is an approximation: $m$ weighted-Jacobi
steps on $A_{FF}X = -A_{FC}$ from the geometric predict,
\begin{equation}\label{eq:jacobi}
  X^{(0)} = P_{k,F}, \quad X^{(j+1)} = X^{(j)} + \omega D_{FF}^{-1}\big(-A_{FC} - A_{FF}X^{(j)}\big), \quad
  \widehat\Phi_k = \begin{bmatrix} I\\ X^{(m)}\end{bmatrix},
\end{equation}
with $D_{FF} = \operatorname{diag}(A_{FF})$, $\omega = 2/3$, $m = 4$ (prolongator smoothing in the sense of smoothed aggregation
\cite{vanek1996}). Increasing $m$ approaches the ideal prolongation only when the Jacobi iteration converges, which positive
definiteness alone does not guarantee; the value $m=4$ was fixed during development and its effect is measured in
\appref{app:sensitivity}. Coarse operators are Galerkin, $A_{k-1} = \widehat\Phi_k^{\!\top} A_k\widehat\Phi_k$, residuals are
restricted with $\widehat\Phi_k^{\!\top}$, and the hierarchy is computed per sample from $A$ alone. The statement that
\eqref{eq:jacobi} reproduces constants (or rigid-body modes) where a local condition holds is Proposition~\ref{prop:constants}
of \appref{app:proofs}.

\subsection{The learned cycle}\label{sec:cycle}
For $k = L, L-1, \dots, 1$ and a right-hand side $b$ on level $k$, one cycle returns three parts,
\begin{equation}\label{eq:cycle}
\begin{aligned}
  x &= \mathcal{S}^{\mathrm{pre}}(b; A_k), &
  e_c &= \mathcal{C}_{k-1}\big(\widehat\Phi_k^{\!\top}(b - A_k x)\big), &
  x_c &= \widehat\Phi_k e_c, \\
  x_p &= \mathcal{S}^{\mathrm{post}}\big(b - A_k(x + x_c); A_k\big), &
  \mathcal{C}_k(b) &= x + x_c + x_p, & \mathcal{C}_0 &= A_0^{-1},
\end{aligned}
\end{equation}
with an exact dense solve on the coarsest level (28 nodes for the perforated square). Only the smoothers are learned. A smoother
maps a right-hand side $b$ on a level with operator $A_k$ to a correction. With $d = \operatorname{diag}(A_k)$, $z = b/d$ and
$\sigma = \|z\|_{\mathrm{rms}}$, the node features are $(z_i/\sigma,\ \log(d_i/\operatorname{median} d))$; an encoder lifts them to a
latent $s\in\mathbb{R}^{n\times C}$, and $T$ steps of
\begin{equation}\label{eq:smoother}
\begin{aligned}
  s_i \leftarrow s_i &+ \mathcal{L}_t\big([\,s_i,\ \textstyle\sum_{j\ne i} W_{ij} V_t s_j\,]\big)\\
       &+ \mathcal{M}_t\big([\,s_i,\ \textstyle\sum_{j\ne i} W_{ij} V_t s_j,\ \overline{m}_i\,]\big),
  \qquad W_{ij} = A_{ij}/A_{ii},
\end{aligned}
\end{equation}
follow, where $\mathcal{L}_t$ and $V_t$ are linear, $\mathcal{M}_t$ is a small MLP and $\overline{m}_i$ the mean of edge messages
$\mathrm{MLP}([s_i,s_j,\delta_{ij},|\delta_{ij}|,W_{ij}])$ over the neighbours of $i$, with $\delta_{ij}$ the edge vector in units of
the level's mesh width. The output is $\sigma\cdot\mathrm{dec}(s)$. The operator enters only through $W$ and $d$ of the level on
which the smoother acts; one pre- and one post-smoother are shared by all levels ($T = \nu = 4$ stages; $20{,}672$ parameters). The latent width is
$C = 15$. Every MLP has one hidden layer with GELU activation: the encoder is a linear map of the two features plus an MLP
$2\to2C\to C$; in each step $V_t$ maps $C\to C$, $\mathcal{L}_t$ maps $2C\to C$, the edge-message MLP is $2C+4\to C\to C$ and
$\mathcal{M}_t$ is $3C\to C\to C$; the decoder is a linear map plus an MLP $C\to C\to1$.
Note that $\sigma$ is a global quantity of the input, so the output at a node depends on residual entries outside any fixed
neighbourhood, which matters for the interpretation of the local controls (Section~\ref{sec:controls}).

\subsection{Near-kernel deflation and the Galerkin output}\label{sec:deflation}\label{sec:output}
High-contrast media can contain \emph{floating islands}: conductive regions enclosed by insulating barriers and not connected to
$\Gamma_D$. Their indicator vectors $z$ satisfy $z^\top A z/z^\top D z \ll 1$, i.e.\ they lie close to the kernel of $A$, and the
nodal hierarchy used here is not guaranteed to approximate such modes unless its coarse levels resolve the island. We detect
them from $A$ alone. Let $G$ be the graph of strong connections, $-A_{ij}\ge\theta\sqrt{A_{ii}A_{jj}}$ ($\theta = 0.1$). A
connected component of $G$ is \emph{anchored} if one of its rows is strongly coupled to the Dirichlet boundary,
$\sum_j A_{ij}\ge\theta A_{ii}$ (interior rows of a diffusion matrix sum to zero). The indicator of every unanchored component
with Rayleigh quotient $z^\top Az/z^\top Dz<\tau$ ($\tau = 0.05$) is a column of $Z$, and before the cycle
\begin{equation}\label{eq:deflation}
  x_0 = Z(Z^\top A Z)^{-1}Z^\top r .
\end{equation}
If no vector is detected, $x_0 = 0$. The detector is a heuristic for scalar diffusion: its strength criterion presumes
non-positive off-diagonal couplings, which holds for $P_1$ isotropic diffusion on non-obtuse triangulations but is not guaranteed
on a general unstructured mesh, and the anchoring test uses the zero row sums of interior rows. On the barrier test set of the
development data seed it finds $74$ of $75$ geometric islands with precision $75/78$; a small Rayleigh quotient certifies that a
returned vector has low energy, not that every important low-energy direction was found. The thresholds were fixed on the
development seed (a sensitivity check is in \appref{app:sensitivity}); for tensor coefficients we report how many vectors it
returns, and for elasticity it is not used.

The cycle is applied to the deflated residual $b = r - Ax_0$ on the finest level and its parts are combined energy-optimally
together with the near-kernel vectors,
\begin{equation}\label{eq:projection}
  W = [\,x,\ x_c,\ x_p,\ Z\,], \qquad c = x_0 + W y, \qquad W^\top A W\, y = W^\top (r - A x_0),
\end{equation}
with $y = (W^\top AW)^{+}W^\top(r - Ax_0)$ if the columns are dependent (the corrected span is unique). The small Gram system is
formed in double precision after Jacobi scaling, with a ridge and an energy-acceptance test described in \appref{app:impl}.
Training minimises $\ell(c) + \ell(x_0 + x + x_c + x_p)$ with $\ell$ the mean of the relative $L^2$ and energy errors; the second
term keeps each part well scaled, and training through the projection alone stalls on the channel problems. Algorithm~\ref{alg:pass}
summarises a pass, and Figure~\ref{fig:schematic} shows how the prescribed algebraic operations and the two learned smoothers fit
together.
\begin{algorithm}[tbp]
\caption{One OA-VCB pass, ideal construction (exact small solves; the regularised implementation is in \appref{app:impl}).
\emph{Set-up} (once per operator, no learned parameters): hierarchy and coarse factorisation (lines 1--4) and near-kernel vectors
$Z$ (line 5). \emph{Application} (per residual, lines 6--10). $\mathcal{S}^{\mathrm{pre}},\mathcal{S}^{\mathrm{post}}$ are the
learned smoothers \eqref{eq:smoother}, shared by all levels; if $Z$ is empty, $x_0 = 0$.}\label{alg:pass}
\begin{algorithmic}[1]
\Require fine operator $A_L$; geometric prolongations $P_k$; $m = 4$, $\omega = 2/3$
\For{$k = L, \dots, 1$} \State $\widehat\Phi_k \gets$ \eqref{eq:jacobi}; \quad $A_{k-1} \gets \widehat\Phi_k^{\!\top}A_k\widehat\Phi_k$; \quad smoother inputs $d_k$, $W^{(k)}$ of $A_k$ \EndFor
\State factor $A_0$ (dense, exact)
\State $Z \gets$ indicators of the unanchored strong-connection components of $A_L$ with Rayleigh quotient $<\tau$
\Statex
\State $x_0 \gets Z(Z^{\!\top}AZ)^{-1}Z^{\!\top}r$, \quad $b \gets r - Ax_0$ \Comment{residual $r = F - Au_c$}
\State $(x, x_c, x_p) \gets \textsc{Cycle}(L, b)$
\State $W \gets [\,x,\ x_c,\ x_p,\ Z\,]$; solve $W^{\!\top}AW\,y = W^{\!\top}b$ \Comment{Galerkin output \eqref{eq:projection}}
\State \Return $\hat u = u_c + x_0 + Wy$
\Function{Cycle}{$k$, $b$}
  \If{$k = 0$} \Return $A_0^{-1}b$ \EndIf
  \State $x \gets \mathcal{S}^{\mathrm{pre}}(b; A_k)$; \quad $x_c \gets \widehat\Phi_k\,\textsc{Cycle}\big(k-1,\ \widehat\Phi_k^{\!\top}(b - A_kx)\big)$
  \State $x_p \gets \mathcal{S}^{\mathrm{post}}\big(b - A_k(x + x_c); A_k\big)$; \Return $(x, x_c, x_p)$ if $k = L$, else $x + x_c + x_p$
\EndFunction
\end{algorithmic}
\end{algorithm}
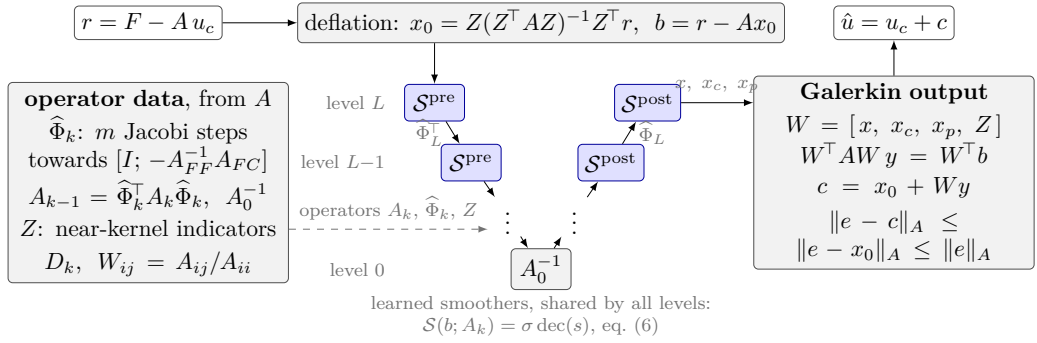
\begin{figure}[tbp]\centering
\resizebox{\linewidth}{!}{\begin{tikzpicture}[x=1cm, y=1cm, font=\footnotesize, >=latex,
  box/.style={draw=black!70, rounded corners=2pt, align=center, inner sep=3pt},
  op/.style={box, fill=black!5},
  learn/.style={box, fill=blue!12, draw=blue!60!black, minimum width=9mm, minimum height=5.5mm},
  lvl/.style={font=\scriptsize, text=black!55}]
\node[box] (in) at (1.80,2.55) {$r = F - A\,u_c$};
\node[op, text width=40mm] (data) at (1.80,0.15) {\textbf{operator data}, from $A$\\[2pt]
  $\widehat\Phi_k$: $m$ Jacobi steps\\ towards $[I;\,-A_{FF}^{-1}A_{FC}]$\\[2pt]
  $A_{k-1}=\widehat\Phi_k^{\!\top}A_k\widehat\Phi_k$, \ $A_0^{-1}$\\[2pt]
  $Z$: near-kernel indicators\\[2pt] $D_k$, \ $W_{ij}=A_{ij}/A_{ii}$};
\node[op] (defl) at (7.70,2.55) {deflation: $x_0 = Z(Z^{\!\top}AZ)^{-1}Z^{\!\top}r$, \ $b = r - Ax_0$};
\node[learn] (s3) at (6.10,1.35) {$\mathcal{S}^{\mathrm{pre}}$};   \node[learn] (p3) at (9.30,1.35) {$\mathcal{S}^{\mathrm{post}}$};
\node[learn] (s2) at (6.65,0.45) {$\mathcal{S}^{\mathrm{pre}}$};  \node[learn] (p2) at (8.75,0.45) {$\mathcal{S}^{\mathrm{post}}$};
\node (s1) at (7.20,-0.35) {$\vdots$};                            \node (p1) at (8.20,-0.35) {$\vdots$};
\node[op, minimum width=9mm] (c0) at (7.70,-1.15) {$A_0^{-1}$};
\draw[->] (s3) -- node[left, lvl] {$\widehat\Phi_L^{\!\top}$} (s2);  \draw[->] (s2) -- (s1);  \draw[->] (s1) -- (c0);
\draw[->] (c0) -- (p1); \draw[->] (p1) -- (p2); \draw[->] (p2) -- node[right, lvl] {$\widehat\Phi_L$} (p3);
\node[lvl, anchor=east] at (5.50,1.35) {level $L$}; \node[lvl, anchor=east] at (5.50,0.45) {level $L{-}1$};
\node[lvl, anchor=east] at (5.50,-1.15) {level $0$};
\draw[->] (defl.south -| s3) -- (s3);
\node[lvl, align=center] at (7.70,-1.85) {learned smoothers, shared by all levels:\\ $\mathcal{S}(b;A_k) = \sigma\,\mathrm{dec}(s)$, eq.~\eqref{eq:smoother}};
\node[op, text width=40mm] (out) at (13.00,0.30) {\textbf{Galerkin output}\\[2pt]
  $W=[\,x,\ x_c,\ x_p,\ Z\,]$\\[1pt] $W^{\!\top}AW\,y = W^{\!\top}b$\\[1pt] $c = x_0 + Wy$\\[3pt]
  $\|e-c\|_A\le\|e-x_0\|_A\le\|e\|_A$};
\node[box] (res) at (13.00,2.55) {$\hat u = u_c + c$};
\draw[->] (in) -- (defl);
\draw[->, dashed, black!55] (data.east |- 0,-0.55) -- node[above, lvl] {operators $A_k$, $\widehat\Phi_k$, $Z$} (6.9,-0.55);
\draw[->] (p3.east) -- node[above, lvl] {$x,\ x_c,\ x_p$} (out.west |- p3.east);
\draw[->] (out) -- (res);
\end{tikzpicture}}
\caption{One OA-VCB pass, ideal construction. Grey boxes denote prescribed algebraic operations; blue boxes denote the learned
smoothers, one pre- and one post-smoother shared by all levels. The operator-data box is precomputed from each sample's
operator; deflation and the Galerkin output depend on the residual and on the cycle's parts. The output is the energy-optimal
combination of the three parts and the near-kernel vectors, which gives the safeguard of Proposition~\ref{prop:safeguard}.}
\label{fig:schematic}
\end{figure}

\subsection{Structural properties}\label{sec:properties}
\begin{proposition}[containment of fixed polynomial smoothers]\label{prop:containment}
Let $\nu\ge1$ stages and $C\ge2$ latent channels be given, and let $p(\mu) = \sum_{j=0}^{\nu-1}a_j\mu^j$ be any real polynomial of
degree at most $\nu-1$. There are encoder, step maps and decoder of both smoothers (explicit in \appref{app:proofs}) such that on
every level each smoother returns
\begin{equation}\label{eq:polysmoother}
\begin{aligned}
  x &= p(B_k)\,D_k^{-1}b, & B_k &= D_k^{-1}A_k,\\
  A_k^{-1}b - x &= R(B_k)\,A_k^{-1}b, & R(\mu) &= 1 - \mu\,p(\mu),
\end{aligned}
\end{equation}
with the same coefficients on all levels, and the network output \eqref{eq:deflation}--\eqref{eq:projection} equals the
classical deflated V$(\nu,\nu)$ cycle with these smoothers, transfers \eqref{eq:jacobi}, Galerkin levels, exact coarsest solve
and Galerkin projection onto $\operatorname{span}\{x,x_c,x_p,Z\}$. Equivalently, every real residual polynomial $R$ of degree at
most $\nu$ with $R(0) = 1$ is represented. In particular the cycle represents (a)~weighted Jacobi smoothing and
(b)~degree-$\nu$ Chebyshev smoothing on a fixed interval $[\lambda/q,\lambda]$, $\lambda>0$, $q>1$, whose residual polynomial
is $T_\nu\big((\theta-\mu)/\delta\big)/T_\nu(\theta/\delta)$ with $\theta = (\lambda+\lambda/q)/2$, $\delta = (\lambda-\lambda/q)/2$
\cite{saad2003}.
\end{proposition}
Containment is a statement of representation. It requires the coefficients to be shared across levels; it bounds the correction
polynomial by degree $\nu-1$ (three at $\nu = 4$) and the residual polynomial by degree $\nu$; it does not assert that the
represented cycle contracts, that a fixed interval encloses the spectrum on every level, or that the learned smoother attains
the best member of the class; and the learned nonlinear smoother has no polynomial degree. \emph{Containment is not optimality:}
it does not contain adaptive Chebyshev or Krylov schedules, whose coefficients depend on the level or the input, and it does not
imply that a smoother outside the class cannot be represented. Numerically the network with the weights of the construction
reproduces the classical cycles to about $10^{-10}$ (\appref{app:impl}).
\begin{proposition}[positive homogeneity]\label{prop:homogeneity}
$\mathcal{N}(\lambda r; A) = \lambda\,\mathcal{N}(r; A)$ for every $\lambda > 0$.
\end{proposition}
Homogeneity concerns the scale of the data only; it says nothing about right-hand sides of a different shape
(Section~\ref{sec:loads}) and does not cover $\lambda<0$, since the smoothers are not odd in $r$. In double precision
$|\mathcal{N}(2.5r)-2.5\mathcal{N}(r)|/|2.5\mathcal{N}(r)|\le 7\times10^{-13}$, and $7\times10^{-7}$ in the single-precision
implementation for residual scales from $10^{-20}$ to $10^{8}$ (\appref{app:impl}).
\begin{proposition}[energy-norm safeguard]\label{prop:safeguard}
For any network parameters, in exact arithmetic and with $e = A^{-1}r$,
\[
  \|e - c\|_A \;=\; \min_{y}\|e - x_0 - Wy\|_A \;\le\; \|e - x_0\|_A \;=\; \min_{w\in\operatorname{span}Z}\|e - w\|_A \;\le\; \|e\|_A ,
\]
so that $\|u - (u_c + c)\|_A \le \|u - u_c\|_A$: the output is never worse in energy than the coarse solution it corrects, nor than
deflation alone.
\end{proposition}
The argument is the Ritz/Galerkin step of deflated and flexible Krylov methods \cite{frank2001,rudikov2024}; it bounds the damage
relative to $u_c$ and says nothing about the gain. It is a statement about exact arithmetic. The implementation adds a ridge of
$10^{-10}$ and an energy-acceptance test, which preserve energy non-increase up to rounding but not the minimisation equality
(\appref{app:impl}); every trained model is in addition evaluated for the number of test samples on which it is worse than $u_c$.
The operator data cost is measured in Section~\ref{sec:cost}: because the $m$ Jacobi steps act on Galerkin operators that widen
from level to level, the operator complexity $\sum_k\operatorname{nnz}(A_k)/\operatorname{nnz}(A_L)$ is $4.5$ for three levels
and $8.6$ for four.

\section{Controls and experimental design}\label{sec:controls}

\subsection{Problems and data}\label{sec:problems}
All experiments use the perforated unit square (three circular holes, homogeneous Dirichlet conditions on the holes and on the
outer boundary), meshed by an unstructured coarse triangulation that is refined uniformly. The main configuration has $L=3$
refinements ($3{,}486$ unknowns) and a level-$k_0=2$ Galerkin coarse solve $u_c$. Two isotropic coefficient classes test
different difficulties. \emph{Channels}: four meandering conductive bands of width $1/16$ crossing the domain, $a = 10^4$ in the
bands and $1$ elsewhere, load $f=1$; the near-kernel difficulty is absent and the cycle must resolve contrast along the bands.
\emph{Barriers}: the same geometry with $a = 10^{-4}$ in the bands, which insulate and can enclose floating conductive islands;
the source $f=1$ acts only where $a\ge 1/2$. Figure~\ref{fig:setup} shows the meshes and typical fields.
\begin{figure}[tbp]\centering\includegraphics[width=0.92\linewidth]{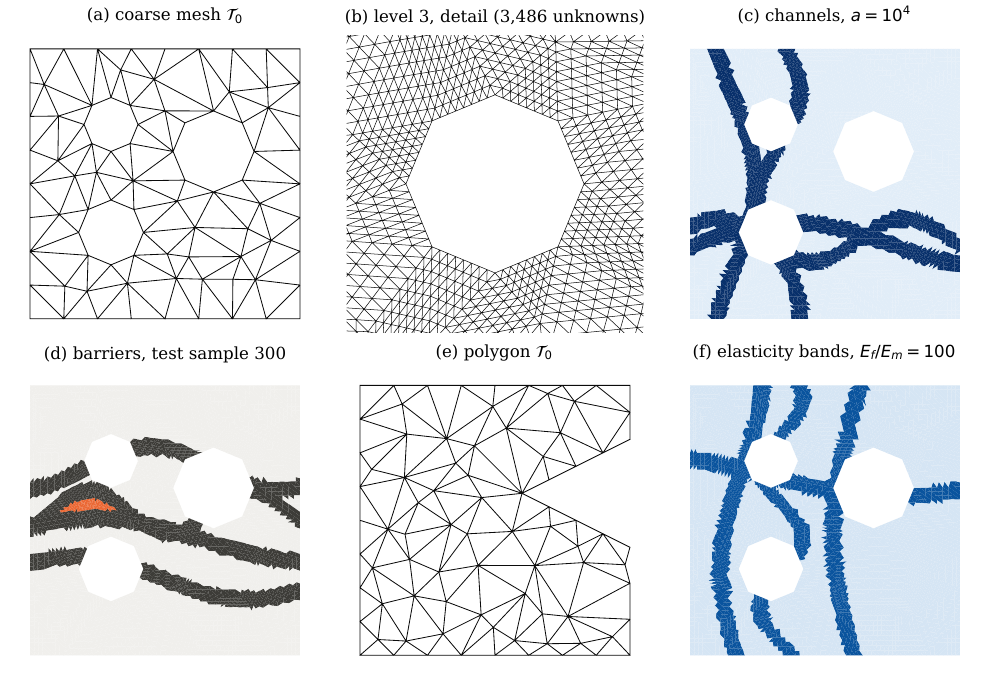}
\caption{Meshes and coefficient samples. (a) Unstructured coarse mesh $\mathcal{T}_0$ of the perforated square; (b) detail
of its third uniform refinement; (c) channels, conductive bands $a = 10^4$ (dark) in a unit background; (d) barriers,
insulating bands $a = 10^{-4}$ (dark), with a floating conductive island (orange, geometric label); (e) coarse mesh of the
polygonal domain; (f) elasticity, Young's modulus $100$ times larger in the bands. Field (c) is training sample 0 of held-out
data seed 1; field (d) is test sample 300 of the same seed, chosen as the first test sample whose floating island has more than
30 elements; (f) is training sample 0 of the elasticity data set.}\label{fig:setup}
\end{figure}
Each coefficient set consists of independent draws; one network is trained per data set. Every learned model in the main
comparisons has $20{,}672$ parameters (OA-VCB), is trained with AdamW and a one-cycle schedule (peak learning rate
$2\times10^{-3}$), batch size $8$ and $300$ epochs on the mean of the relative $\ell^2$ and energy errors, and is not tuned per
coefficient class. The quantity of interest is the algebraic error with respect to the discrete solution $u$ of $Au=F$, mainly in
the energy norm; discretisation error is not part of any number.

\subsection{Evidence sets and their status}\label{sec:evidence}
Two sets of evidence are used, and they are never pooled.
\emph{Original study (13 October).} For each data seed $s$, $400$ coefficient fields are drawn from the random seed $1000+s$;
samples $0$--$299$ train and $300$--$399$ test. Data seed $0$ was the development seed; the results are on the held-out seeds
$1$--$3$ (100 test operators each) unless stated otherwise. Each model is trained separately on each data seed with
initialisation seed $0$, and OA-VCB was retrained with initialisation seeds $1$ and $2$.
\emph{Stage 1.} Before any confirmation operator was generated, a written protocol fixed the comparators, the primary
comparisons, the descriptive bands for the primary comparison and the statistics; later changes were recorded as dated amendments
before the numbers they affect were computed (chronology and hashes in \appref{app:stage1}). Five kinds of operator set are used.
\emph{T}: the original training operators, used only to fit fixed polynomials and to choose a Chebyshev interval ratio.
\emph{Development set D}: the original test operators of data seed 1, inspected before Stage 1 and used only for implementation
checks and exploratory analyses, never for confirmation. \emph{Confirmation set C}: 100 operators per family generated after the
protocol was fixed; the pre-declared comparisons were evaluated here, once. \emph{Set LC}: the first 50 channel operators of C with
four off-training load families. \emph{Reserve set}: 100 further operators per family, generated after a second amendment had fixed
the endpoints R1--R4 (the channel bracket, the barrier bracket, the comparison of Jacobi-PCG(14) with SGS-PCG(4), and the local $K_5$
criterion) and specified the repaired timings and the operation ledger. The Stage 1 results reported here evaluate stored
data-seed-1 checkpoints from the original study, including three fixed-load OA-VCB initialisations, residual-linear controls, and
the separately trained random-load model used in the load challenge. No neural network was trained or retrained for these
evaluations; polynomial comparator coefficients were fitted on T. Three initialisations evaluated on the same 100 operators are not
300 independent operators, and all bootstrap intervals are pointwise, conditional on the trained models, and describe variability
over test operators only. One protocol step was not carried out as specified: a run-time check of the frozen scale file was not
executed. The file's SHA-256 matches its pre-confirmation hash, but that comparison was made retrospectively and does not prove the
file's contents at run time (\appref{app:stage1}).

\subsection{Metrics}\label{sec:metrics}
Comparisons are paired per operator: for a comparator $X$ and a model $Y$,
\[
  G(X\!\to\!Y) = \exp\Big(\tfrac1n\textstyle\sum_i \log\big(e^{X}_i/e^{Y}_i\big)\Big),\qquad
  W = \tfrac1n\#\{i: e^{Y}_i<e^{X}_i\},
\]
the geometric-mean ratio of the energy errors and the fraction of operators on which $Y$ is more accurate, so $G>1$ favours $Y$.
Intervals are $95\%$ percentile bootstrap intervals over operators ($5{,}000$ resamples for the original study, $10{,}000$ for
Stage 1). Arithmetic mean errors are reported alongside in the original tables; where the two summaries disagree we say so. We
write $G$ for the original study and, in Stage 1 tables, $G(X\!\to\!\text{OA-VCB})$ with the comparator named.

\subsection{Comparators}\label{sec:comparators}
Every comparator below uses the same transfers, Galerkin levels, coarse solution $u_c$, deflation vectors, exact coarsest solve and
Galerkin output as OA-VCB, with only the smoothers replaced, unless it is called a whole-system solver.
\emph{Chebyshev reference.} The deflated V$(4,4)$ cycle with degree-four Chebyshev smoothers on the interval $[\lambda_k/10,\lambda_k]$,
with $\lambda_k = 3$ on the finest scalar level and otherwise $1.05$ times a $30$-step power-iteration estimate of the largest
eigenvalue of $D_k^{-1}A_k$ (a heuristic estimate, not a certified bound). Because its intervals depend on level and operator,
Proposition~\ref{prop:containment} does not establish that it is contained in the architecture; numerically it is within a few per
cent of the contained fixed-interval Chebyshev member (paired ratio $0.975$--$1.056$, \appref{app:impl}), and the contained
Jacobi member is $1.35$--$1.82$ times less accurate. This is the reference of the original study.
\emph{Jacobi-PCG($s$) and SGS-PCG($s$) smoothing calls.} In every pre- and post-smoothing call the smoother is replaced by $s$
steps of conjugate gradients on the call's input, started from zero, with the Jacobi (respectively symmetric Gauss--Seidel,
natural ordering, one symmetric sweep per application) preconditioner. $s=4$ matches the four learned stages. These are controls
\emph{inside} the cycle; they are not the whole-system PCG solvers below.
\emph{Fitted fixed polynomial (POLY-S).} A degree-three correction polynomial in $D^{-1}A$, with one coefficient set for the
pre-smoothers and one for the post-smoothers shared by all levels (eight numbers), fitted on the training operators to minimise
the evaluation metric (which favours the control). It is a member of the contained class of Proposition~\ref{prop:containment}.
\emph{Per-instance polynomial candidate.} For each operator, separate degree-three polynomials for the six smoothing calls of the
cycle are fitted with the exact solution from five starting points by L-BFGS. Its feasible set contains, in exact arithmetic, the
trajectory of every smoother whose output lies in the Krylov space $K_4$ of its own call's input. The fitted error is an upper
bound on the best member of that class, not a certified optimum: a candidate that beats OA-VCB refutes superiority over the class
on that operator, whereas a candidate that loses proves nothing about the class.
\emph{Local Krylov-space diagnostic.} For the input $b$ of one fine-level smoothing call, let
$K_d=\operatorname{span}\{D^{-1}b,\,(D^{-1}A)D^{-1}b,\dots,(D^{-1}A)^{d-1}D^{-1}b\}$. The ratio $\rho_{K_d}=E_{\mathrm{NN}}/E_{K_d}$ divides
the energy error of the network smoother's output by that of the exact energy-norm minimiser over $K_d$, computed from a
rank-aware $A$-orthonormal basis (on the reserve set the $d$-step Jacobi-PCG error equals it to a relative $3\times10^{-15}$). Three facts bound its
interpretation. The input is \emph{fixed}: a different pre-smoother changes the later inputs of the coarse correction and the
post-smoother, so a local optimum is not a bound for complete cycles. $K_4$ is the class of the contained correction polynomials
(degree three), and $K_5$ is a \emph{degree control} with matrix-propagation degree matching four message-passing layers; it is not an
equal-information or equal-cost control, because the network normalises by a global scale and uses nonlinear features. And leaving a
scalar polynomial space does not show that nonlinearity is necessary, since a spatially varying residual-linear operator can also
leave it.
\emph{Whole-system solvers.} Two-grid PCG (weighted Jacobi$(1,1)$ smoothing and an exact level-$(L-1)$ Galerkin coarse solve),
smoothed-aggregation AMG-preconditioned CG (PyAMG), and a sparse direct solve (SuperLU), all started from $u_c$ except the direct
solve. For the iterative solvers the iteration count that first reaches OA-VCB's error on each operator is selected offline, with
the reference solution, outside the timing (an oracle stopping rule, not a deployable one).

\subsection{Cost measurement and disclosure}\label{sec:costmeasure}
Costs are reported on one CPU in two measurements that are never combined: the original platform (an Intel Xeon at $2.1$\,GHz,
$2$ threads, L3 and L4) and Stage 1 (an Intel Xeon at $2.10$\,GHz, two threads, five repetitions; repaired timings after review,
with an approximate operation ledger; \appref{app:stage1}). A third machine (an Intel Xeon at $2.80$\,GHz) was used for the
original larger-budget measurement of Figure~\ref{fig:avt}.
Python implementations are compared; the learned cycle runs in single precision inside double-precision deflation and output,
whereas the classical cycles run in double precision, and no precision ablation was run.
Software for this study (data generation, models, training, evaluation, classical comparators, numerical checks and the scripts
that aggregate stored records into tables and figures) was written and run with the assistance of generative AI systems under the
author's direction (see the declaration). No reported number, table or data plot was produced by these tools directly: all are
computed by scripts from stored per-sample records, the principal summaries were recomputed by independent scripts, and the
Stage 1 statistics were independently recomputed from the per-operator errors without the study's own statistics code
(agreement to $9\times10^{-16}$).
Learned operator baselines of similar size, adapted to this setting, were also trained in the original study; they are not part
of the comparisons of this paper and are kept in the released archive.

\section{Results: local gains, full-cycle accuracy and cost}\label{sec:results}

\subsection{Full-cycle accuracy against stronger one-pass controls}\label{sec:full}\label{sec:accuracy}
\begin{table}[tbp]\centering\footnotesize\setlength{\tabcolsep}{3.5pt}
\caption{One-pass energy error of the classical comparator divided by that of OA-VCB (paired geometric-mean $G$, pointwise 95\% bootstrap interval over operators; $G>1$ favours OA-VCB). \emph{Panel A} (Stage 1 set C): 100 operators generated for the confirmation, one fixed-load network per initialisation seed; every comparator is the same cycle with only the smoothers replaced. ``Jacobi-PCG(4)\textquoteright\textquoteright{} is a four-step Jacobi-preconditioned conjugate-gradient smoother \emph{inside every smoothing call}, not the whole-system two-grid PCG solver of Section~\ref{sec:cost}. ``Polynomial candidate\textquoteright\textquoteright{} is the best per-instance fit of separate degree-three polynomials for the six smoothing calls, obtained with the exact solution and five starts; its error is an upper bound on the error of the best class member, not a certified optimum. ``SGS-PCG(4)\textquoteright\textquoteright{} uses the symmetric Gauss--Seidel preconditioner. \emph{Panel B} (original test sets, 13 Oct run, initialisation 0): 100 test operators per data seed, Chebyshev reference only. The two panels use different operators and checkpoints and are never pooled.}\label{tab:s1main}
\begin{tabular}{>{\raggedright\arraybackslash}p{4.6cm}lll}\toprule
\multicolumn{4}{l}{\emph{A. Stage 1 set C, channels (contrast $10^4$); columns: OA-VCB initialisation seed}} \\
comparator & 0 & 1 & 2 \\\midrule
Jacobi-PCG(4) smoothing calls (primary, pre-declared) & 5.70 [5.25, 6.13] & 5.94 [5.56, 6.32] & 6.12 [5.71, 6.54] \\
Chebyshev $\nu=4$ (reference of the original study) & 9.18 [8.43, 9.94] & 9.58 [8.92, 10.23] & 9.87 [9.16, 10.59] \\
SGS-PCG(4) smoothing calls & 1.16 [1.08, 1.23] & 1.21 [1.15, 1.27] & 1.24 [1.18, 1.31] \\
fitted fixed polynomial, degree 3 (contained class) & 5.14 [4.79, 5.51] & 5.37 [4.96, 5.81] & 5.53 [5.13, 5.95] \\
polynomial candidate (per instance, upper bound) & 4.14 [3.82, 4.45] & -- & -- \\\midrule
\multicolumn{4}{l}{\emph{A. Stage 1 set C, barriers (contrast $10^{-4}$); columns: OA-VCB initialisation seed}} \\
comparator & 0 & 1 & 2 \\\midrule
Jacobi-PCG(4) smoothing calls (primary, pre-declared) & 1.45 [1.33, 1.59] & 1.42 [1.28, 1.58] & 1.38 [1.24, 1.54] \\
Chebyshev $\nu=4$ (reference of the original study) & 2.00 [1.82, 2.19] & 1.96 [1.76, 2.17] & 1.90 [1.70, 2.11] \\
SGS-PCG(4) smoothing calls & 0.58 [0.50, 0.68] & 0.57 [0.48, 0.67] & 0.55 [0.47, 0.65] \\
fitted fixed polynomial, degree 3 (contained class) & 1.76 [1.60, 1.93] & 1.73 [1.55, 1.91] & 1.67 [1.51, 1.85] \\
polynomial candidate (per instance, upper bound) & 0.75 [0.68, 0.83] & -- & -- \\\midrule
\multicolumn{4}{l}{\emph{B. Original test sets (13 Oct), Chebyshev $\nu=4$ only; columns: data seed}} \\
family & 1 & 2 & 3 \\\midrule
channels & 10.23 [9.75, 10.71] & 9.76 [9.16, 10.31] & 9.68 [9.12, 10.24] \\
barriers & 2.13 [1.90, 2.38] & 1.93 [1.70, 2.18] & 1.93 [1.70, 2.19] \\
\bottomrule\end{tabular}\end{table}

The pre-declared primary comparison of Stage 1 replaces the four-stage smoothers by four steps of Jacobi-preconditioned conjugate
gradients inside each smoothing call (Section~\ref{sec:comparators}; Table~\ref{tab:s1main}, panel A). On the confirmation set,
OA-VCB is $5.70$, $5.94$ and $6.12$ times more accurate than this control on channels for the three initialisations, with
$W = 0.99$--$1.00$, and $1.45$, $1.42$ and $1.38$ times on barriers with $W = 0.76$--$0.82$. The frozen labels read ``not
accounted for by the tested four-step Jacobi-PCG control'' on channels and ``partly accounted for / inconclusive'' on barriers
(the lower interval limits, $1.24$--$1.33$, lie below the declared threshold of two); they concern one algorithm, not every Krylov
method. The control is itself $1.61$ [$1.58$, $1.64$] (channels) and $1.37$ [$1.32$, $1.43$] (barriers) times more accurate than the
Chebyshev reference, so the Chebyshev ratios on the same operators are larger ($9.18$--$9.87$ and $1.90$--$2.00$); they reproduce
the original study (panel B, different operators and checkpoints, never pooled with panel A; $9.68$--$10.74$ and $1.80$--$2.13$
over three training initialisations, \appref{app:repeat}). A fitted fixed polynomial of the contained class is overtaken by
$5.14$--$5.53$ on channels and $1.67$--$1.76$ on barriers.

The picture on barriers is weaker and mixed. Against Gauss--Seidel-preconditioned CG with four steps per call the ordering
reverses: $G = 0.55$--$0.58$, with OA-VCB more accurate on only about a quarter of the operators, while on channels it keeps a
small margin ($1.16$--$1.24$). The best per-instance polynomial candidate found, which uses the exact solution and is an upper
bound on the error of the best class member (Section~\ref{sec:comparators}), is $4.14$ [$3.82$, $4.45$] times less accurate than OA-VCB on
channels, but on barriers it is \emph{more} accurate: $G = 0.75$ [$0.68$, $0.83$], with OA-VCB better on $34\%$ of the operators
and, on island operators, $0.57$ [$0.47$, $0.69$] (Table~\ref{tab:s1sub}). A feasible candidate that wins refutes superiority over
the polynomial class on those instances; a candidate that loses, as on channels, does not establish superiority of OA-VCB over
the class optimum. The full ladder of
comparators is in \appref{app:stage1}.

\subsection{A local diagnostic}\label{sec:local}
\begin{table}[tbp]\centering\footnotesize\setlength{\tabcolsep}{4pt}
\caption{Local Krylov-space diagnostic on 100 fresh channel operators (Stage 1 reserve set). For the input $b$ of one fine-level smoothing call, $\rho_{K_d}=E_{\mathrm{NN}}/E_{K_d}$ is the energy error of the network smoother divided by the exact energy-norm minimum over the Jacobi-preconditioned Krylov space $K_d=K_d(D^{-1}A,D^{-1}b)=\operatorname{span}\{D^{-1}b,(D^{-1}A)D^{-1}b,\dots,(D^{-1}A)^{d-1}D^{-1}b\}$ of the \emph{same} fixed input (rank-aware $A$-orthonormal basis). Values below one favour the network. Pre- and post-smoother inputs of the same complete cycle are listed separately; $\rho_{K_5}$ is a degree control (matrix-propagation degree four plus one), not an equal-information or equal-cost control, because the network normalises by a global residual scale and uses nonlinear features. Counts are numbers of the 100 inputs. The diagnostic concerns one fixed input and does not describe complete cycles, whose later inputs differ.}\label{tab:s1local}
\begin{tabular}{llllllll}\toprule
 & & $K_4$ & \multicolumn{2}{c}{$K_5$} & \multicolumn{3}{c}{$K_8$} \\ \cmidrule(lr){3-3}\cmidrule(lr){4-5}\cmidrule(lr){6-8}
init & input & median & median (10--90\%) & $\rho\le0.99$ & median & $\rho\le0.99$ & $\rho<1$ \\\midrule
0 & pre & 0.610 & 0.709 (0.662--0.768) & 100 & 1.025 & 19 & 28 \\
0 & post & 0.488 & 0.607 (0.518--0.767) & 97 & 1.069 & 23 & 27 \\
1 & pre & 0.601 & 0.695 (0.649--0.763) & 100 & 1.011 & 36 & 40 \\
1 & post & 0.488 & 0.614 (0.528--0.779) & 99 & 1.061 & 21 & 22 \\
2 & pre & 0.595 & 0.689 (0.640--0.755) & 100 & 1.002 & 43 & 49 \\
2 & post & 0.471 & 0.595 (0.475--0.835) & 98 & 1.049 & 35 & 37 \\
\bottomrule\end{tabular}\end{table}

The comparators above change the later inputs of the cycle. To ask a cleaner question we fix the input of one smoothing call and
compare the network's output with the exact energy minimum over a Krylov space of that input (Section~\ref{sec:comparators}).
On the confirmation set the fine-level smoother improves on the $K_4$ minimum for all $100$ channel inputs at both positions
(median $\rho_{K_4}=0.614$ before and $0.491$ after the coarse correction). On $100$ fresh reserve channel operators
(Table~\ref{tab:s1local}), the improvement persists against $K_5$: $\rho_{K_5}\le0.99$ on $97$--$100$ of $100$ inputs for every
combination of initialisation and position, with median $\rho_{K_5}$ of $0.689$--$0.709$ before and $0.595$--$0.614$ after the
coarse correction, so the frozen criterion ($\ge90$ inputs) is met in all six cases. The orthonormality defect of the Krylov
bases is below $10^{-14}$, and the Jacobi-PCG iterate reproduces the minimiser to $3\times10^{-15}$. Against $K_8$ the median ratio
reverses ($1.002$--$1.025$ before, $1.049$--$1.069$ after), the network is at least $1\%$ better on only $19$--$43$ of $100$ inputs
and better at all on $22$--$49$: the diagnostic does not say that the network dominates a larger local space, nor that $K_8$
dominates it on every input. On barriers the local gain is position dependent (confirmation set: $\rho_{K_4}\le0.99$ on $44$ of $100$
inputs before and $94$ of $100$ after the coarse correction).

This is a rank-checked numerical evaluation of an energy-minimising space for a fixed input. It is neither a theorem about all
operators nor a statement about complete cycles. It does not show that the learned smoother communicates better per stage, since
$K_5$ is a degree control only, and it does not show that nonlinearity is necessary: a spatially varying residual-linear operator
can also leave the scalar polynomial space. It is a local diagnostic of what the output contains that a fixed-input polynomial
space of the same degree does not.

\subsection{Larger classical budgets in the complete cycle}\label{sec:budget}
\begin{table}[tbp]\centering\footnotesize\setlength{\tabcolsep}{3.5pt}
\caption{Larger classical smoothing budgets in the complete cycle (Stage 1 reserve set; 100 fresh operators per family; generated after the amendment that fixed these endpoints). $G$ is the paired geometric-mean ratio of the energy error of the cycle with Jacobi-PCG($s$) smoothing calls ($s$ steps in each pre- and post-smoothing call, all other ingredients unchanged) to that of OA-VCB, with the pointwise 95\% bootstrap interval; values below one mean that the classical cycle is more accurate. The bracket concerns paired geometric-mean error over operators, not each operator and not an exact crossover step count. Rows within a family are different trained networks evaluated on the same 100 operators.}\label{tab:s1brackets}
\begin{tabular}{p{4.2cm}lll}\toprule
\multicolumn{4}{l}{\emph{channels; columns: OA-VCB initialisation seed}} \\
comparator & 0 & 1 & 2 \\\midrule
Jacobi-PCG(4) & 5.96 [5.41, 6.41] & 6.20 [5.80, 6.57] & 6.43 [6.05, 6.78] \\
Jacobi-PCG(5) & 3.65 [3.32, 3.92] & 3.80 [3.56, 4.02] & 3.94 [3.72, 4.15] \\
Jacobi-PCG(8) & 1.40 [1.28, 1.50] & 1.45 [1.37, 1.53] & 1.51 [1.43, 1.58] \\
Jacobi-PCG(14) & 0.37 [0.34, 0.40] & 0.39 [0.36, 0.41] & 0.40 [0.38, 0.43] \\
SGS-PCG(4) & 1.16 [1.06, 1.24] & 1.21 [1.14, 1.27] & 1.25 [1.19, 1.31] \\
bracket confirmed between & PCG(8)--PCG(14) & PCG(8)--PCG(14) & PCG(8)--PCG(14) \\
\midrule
\multicolumn{4}{l}{\emph{barriers; columns: OA-VCB initialisation seed}} \\
comparator & 0 & 1 & 2 \\\midrule
Jacobi-PCG(4) & 1.54 [1.39, 1.70] & 1.60 [1.45, 1.77] & 1.55 [1.39, 1.72] \\
Jacobi-PCG(5) & 1.21 [1.08, 1.34] & 1.26 [1.13, 1.40] & 1.21 [1.09, 1.36] \\
Jacobi-PCG(8) & 0.72 [0.62, 0.84] & 0.76 [0.66, 0.87] & 0.73 [0.63, 0.84] \\
Jacobi-PCG(14) & 0.35 [0.29, 0.43] & 0.37 [0.30, 0.45] & 0.36 [0.29, 0.44] \\
SGS-PCG(4) & 0.64 [0.55, 0.75] & 0.67 [0.57, 0.78] & 0.64 [0.55, 0.75] \\
bracket confirmed between & PCG(5)--PCG(8) & PCG(5)--PCG(8) & PCG(5)--PCG(8) \\
\bottomrule\end{tabular}\end{table}

\begin{figure}[tbp]\centering\includegraphics[width=\linewidth]{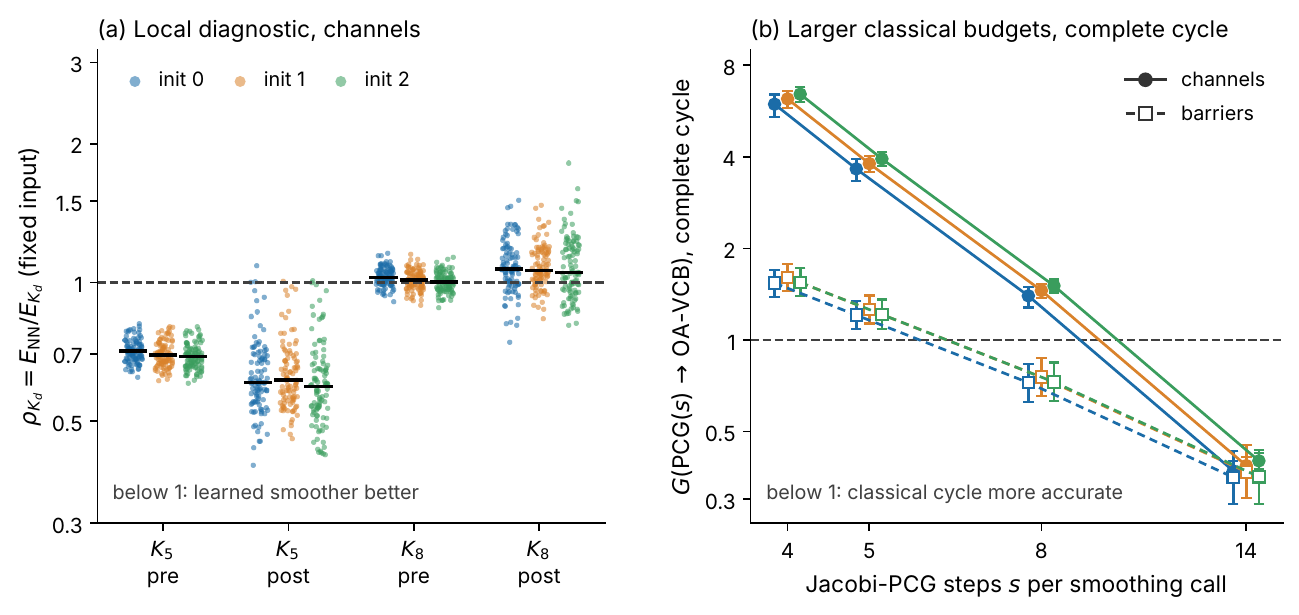}
\caption{Local and full-cycle evidence side by side; the two panels have different inputs and different meanings. (a)~Local
diagnostic on $100$ fresh reserve channel operators: each dot is the input of one fine-level smoothing call, with the ratio of the
network smoother's energy error to the exact energy minimum over $K_5$ or $K_8$ of the \emph{same} input (below $1$: the network
is better for that input; bars: medians; colours: the three training initialisations). (b)~Complete cycles on the same reserve
operators: paired geometric-mean ratio of the energy error of the cycle with $s$ Jacobi-PCG steps per smoothing call to that of
OA-VCB, with pointwise 95\% bootstrap intervals (below $1$: the classical cycle is more accurate; whose later inputs differ from
OA-VCB's). Channels: filled markers, solid lines; barriers: open markers, dashed lines.}\label{fig:keyfig}
\end{figure}
Giving the classical cycle more work closes the gap quickly (Table~\ref{tab:s1brackets}, Figure~\ref{fig:keyfig}b). On the reserve
operators, Jacobi-PCG with $5$ steps per call is still $3.65$--$3.94$ times less accurate than OA-VCB on channels and $1.21$--$1.26$ on
barriers; with $8$ steps the ratio is $1.40$--$1.51$ on channels (every lower limit above $1.279$) and $0.72$--$0.76$ on
barriers; with $14$ steps the classical cycle is the more accurate one on both: $0.37$--$0.40$ on channels (every upper limit
below $0.426$) and $0.35$--$0.37$ on barriers. Bracketing was confirmed for all three initialisations, between $8$ and $14$ steps on
channels (primary endpoint) and between $5$ and $8$ on barriers (secondary endpoint). The brackets concern paired geometric-mean
error over operators; they are not statements about each operator, and they do not locate an exact crossover step count. The
$14$-step classical cycle is $2.5$--$2.7$ times more accurate than OA-VCB on channels and $2.7$--$2.8$ on barriers, and four steps
of Gauss--Seidel-preconditioned CG is overtaken by the $14$-step Jacobi-PCG cycle ($0.32$ and $0.55$). The local result and the
bracket coexist without contradiction: a smoother that is locally better than $K_5$ for its own input does not make the complete
cycle better than a classical cycle with a larger budget.

The original study showed the same budget dependence against the Chebyshev reference on its own test set (data seed $1$, CPU of
a second platform; \appref{app:budget}, Figure~\ref{fig:avt}). The margin on channels falls from $10.2$ at degree four to $4.2$ at
degree eight and $1.8$ at degree sixteen. On barriers sixteen-step Chebyshev reverses the paired advantage ($G = 0.82$, the
classical pass being more accurate on $55$ of $100$ operators), although OA-VCB keeps the smaller arithmetic mean error ($0.0436$
against $0.0587$). Repeating the degree-four reference pass is more effective still: twice it is within a factor $1.30$ of OA-VCB
on channels and $3.5$ times more accurate on barriers, three times it is more accurate by the paired metric on both
($G = 0.19$ and $0.05$), at $77$--$88$\,ms per sample against $268$--$273$\,ms for OA-VCB on that machine. These Chebyshev
comparisons exceed the residual-polynomial degree covered by Proposition~\ref{prop:containment} for four stages and no learned
model with the larger budget was trained; they test a larger classical smoothing budget, not a trained alternative. The two
sets of budget results differ in dataset, checkpoints, smoother family and platform and are not combined.
\begin{figure}[tbp]\centering\includegraphics[width=\linewidth]{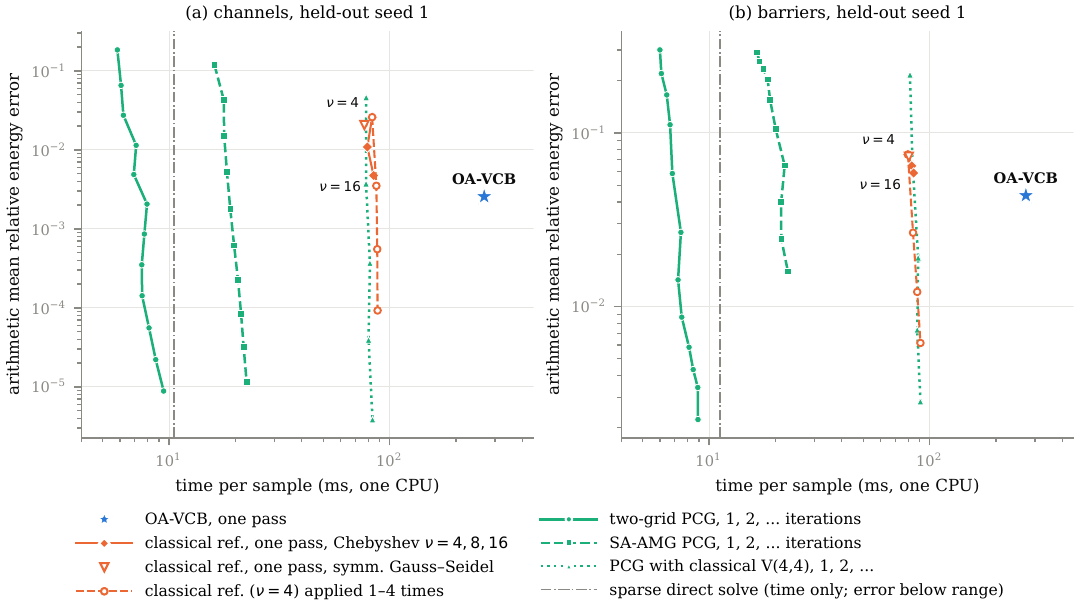}
\caption{Original-study measurement on one CPU (held-out data seed 1, all methods in one run): arithmetic mean relative energy
error against mean time per sample from $u_c$, set-up included. Classical reference: one pass with Chebyshev smoothing of degree
$\nu=4$, $8$, $16$ or symmetric Gauss--Seidel, and the degree-four reference applied one to four times. PCG curves: one point per
iteration. The vertical line marks only the time of the sparse direct solve, whose error is at solver accuracy, below the plotted
range. The paired $G$ of Table~\ref{tab:avt} is a different summary of the same runs.}\label{fig:avt}
\end{figure}

\subsection{Complete cost}\label{sec:cost}
\begin{table}[tbp]\centering\footnotesize\setlength{\tabcolsep}{4pt}
\caption{Complete cost of the Stage 1 comparison (one CPU, one run, Python implementations; channels unless stated; the platform differs from that of Table~\ref{tab:cost}, and the two are not combined). \emph{Panel A}: approximate MAC-equivalent arithmetic of one pass of the complete cycle under the operation-count convention of \appref{app:stage1}, excluding the operator-adapted hierarchy and $u_c$ (common to all rows) and method-specific set-up such as the Chebyshev eigenvalue bounds; it also omits the small dense solves of the Galerkin wrapper and does not enumerate every Gram symmetrisation, scaling and masking operation, and it counts the unpadded wrapper although the executed network wrapper carries padded near-kernel columns. The dense-layer and cycle sparse-product counts agree with independent counters, which does not certify every executed operation. The ledger is neither a time nor a communication volume. \emph{Panel B}: median over 50 operators (ms) of separately timed components; ``to OA-VCB error\textquoteright{}\textquoteright{} selects, offline and per operator, the first iteration whose energy error reaches the one-pass OA-VCB error (an oracle stopping rule, not a deployable one); ``estimated end-to-end\textquoteright{}\textquoteright{} is the median over operators of the per-operator sum of that time and the operator\textquoteright{}s own (median-of-five) $u_c$ construction time. Timings: one Intel Xeon processor at 2.10\,GHz, two threads, Python~3.13, PyTorch~2.14 (CPU), five repetitions. The direct solve is never charged $u_c$.}\label{tab:s1cost}
\begin{tabular}{>{\raggedright\arraybackslash}p{6.9cm}rrr}\toprule
\multicolumn{4}{l}{\emph{A. Counted arithmetic per pass, set-up excluded (channels)}} \\ method & MMAC & SpMV-eq. & OA-VCB/method \\\midrule
Jacobi-PCG(4) smoothing calls & 1.8 & 75 & 406 \\
Jacobi-PCG(8) & 2.8 & 120 & 253 \\
Jacobi-PCG(14) & 4.4 & 188 & 161 \\
SGS-PCG(4) & 3.9 & 167 & 182 \\
Chebyshev $\nu=4$ (pass only) & 1.7 & 73 & 414 \\
OA-VCB (four learned stages) & 712.8 & 30330 & 1 \\
\midrule\multicolumn{4}{l}{\emph{B. Time per operator, ms (median over 50)}} \\ quantity & channels & barriers & \\\midrule
OA-VCB one pass, batch 1, given $u_c$ & 284 & 282 & \\
classical one pass, any tested smoother, given $u_c$ & 66--72 & 66--71 & \\
two-grid PCG to OA-VCB error, given $u_c$ & 4.3 & 4.0 & \\
two-grid PCG to OA-VCB error, estimated end-to-end & 6.9 & 6.6 & \\
SA-AMG PCG to OA-VCB error, given $u_c$ & 14.2 & 14.5 & \\
SA-AMG PCG to OA-VCB error, estimated end-to-end & 16.8 & 17.1 & \\
sparse direct solve (to solver accuracy) & 10.1 & 10.0 & \\
\bottomrule\end{tabular}\end{table}

The accuracy gains do not buy a computational advantage on the measured implementations. Table~\ref{tab:s1cost} gives the Stage 1
account, from the repaired timings and an approximate operation ledger. Under the stated operation-count convention, which counts
one pass with set-up excluded and omits some small wrapper operations (\appref{app:stage1}), one OA-VCB pass counts roughly $713$
million MAC-equivalents ($30{,}330$ fine-SpMV-equivalents) against $1.76$, $2.82$ and $4.41$ million for the Jacobi-PCG cycles with
$4$, $8$ and $14$ steps, ratios of approximately $406$, $253$ and $161$; the $14$-step cycle is the more accurate. The ledger is
neither a time nor a complete end-to-end count, and OA-VCB additionally evaluates $13.1$ million activations and $18$ nonlinear
global reductions per pass. In time, one OA-VCB pass takes a median $284$ (channels) and $282$\,ms (barriers) given $u_c$; a separate
per-component profile attributes about $87$ and $86$\,ms of the pass to operator data. A classical one-pass cycle with any smoother
takes $66$--$72$\,ms, nearly all of it the operator-adapted hierarchy that the learned cycle shares. Two-grid PCG reaches OA-VCB's
per-operator energy error in a median $4.3$ ($4.0$) ms given $u_c$; the median over operators of the per-operator sum with the time
of constructing $u_c$ is $6.9$ ($6.6$)\,ms. Smoothed-aggregation AMG-preconditioned CG takes $14.2$ ($14.5$) and $16.8$
($17.1$)\,ms, and the sparse direct solve, which reaches solver accuracy, $10.1$ ($10.0$)\,ms. The iterative solvers' iteration counts are chosen offline with the true error, so these are oracle
stopping times; a deployable residual-based stopping rule would add convergence-check cost.

The original study measured the same conclusion on another platform (Table~\ref{tab:cost}, \appref{app:budget}). At L3 two-grid PCG,
SA-AMG PCG and a sparse direct solve took $5.7$, $13.5$ and $8.2$\,ms against $165$\,ms for OA-VCB, factors of $29$, $12$ and $20$;
at L4 the factors grew to $98$, $58$ and $50$. From L3 to L4 ($4.1$ times the unknowns) OA-VCB's time grew by a factor $15$ and
two-grid PCG's by $4.3$. The network's CPU inference alone was $17$ (L3) and $63$ (L4) times the full two-grid PCG time including its
set-up, so no number of repeated solves amortises training on the tested configurations. As context only, on the T4 GPU used for
training, inference took $13$\,ms per sample at L3 in batches of eight; optimised implementations, batched GPU operator data and
reuse of an operator for many right-hand sides were not measured. The two measurements are not combined into one time curve.

\section{Limits of the gains, design ablations and scope}\label{sec:limits}

\subsection{Floating islands}\label{sec:islands}
\begin{table}[tbp]\centering\footnotesize\setlength{\tabcolsep}{4pt}
\caption{Barriers split by floating islands (geometric label, conductive component touching no Dirichlet node). Mean energy error OA-VCB / classical reference, paired $G$ and $W$ within each group, and OA-VCB's ratio of island to non-island mean error.}\label{tab:islands}
\begin{adjustbox}{max width=\linewidth}
\begin{tabular}{rrlllll}\toprule
seed & islands & with island & $G$, $W$ & without island & $G$, $W$ & ratio \\\midrule
1 & 39 & 0.100 / 0.167 & 1.86, 0.87 & 0.0075 / 0.0165 & 2.33, 0.95 & 13.4 \\
2 & 48 & 0.139 / 0.162 & 1.34, 0.67 & 0.0058 / 0.0140 & 2.70, 0.98 & 24.1 \\
3 & 40 & 0.143 / 0.188 & 1.56, 0.80 & 0.0056 / 0.0135 & 2.23, 0.95 & 25.6 \\
\bottomrule\end{tabular}
\end{adjustbox}
\end{table}

Floating islands are present in $39$--$48\%$ of the barrier test operators of the original study, and island samples have larger
mean energy errors for every method tested (Table~\ref{tab:islands}). Without islands OA-VCB's paired margin over the reference
is $G = 2.2$--$2.7$ with mean energy error $5.6\times10^{-3}$--$7.5\times10^{-3}$; on island samples the mean error is $13$--$26$ times
larger, the margin falls to $G = 1.3$--$1.9$, and OA-VCB loses to the reference on $13$--$33\%$ of them. Figure~\ref{fig:errormaps}
shows where the remaining error sits.
\begin{figure}[tbp]\centering\includegraphics[width=\linewidth]{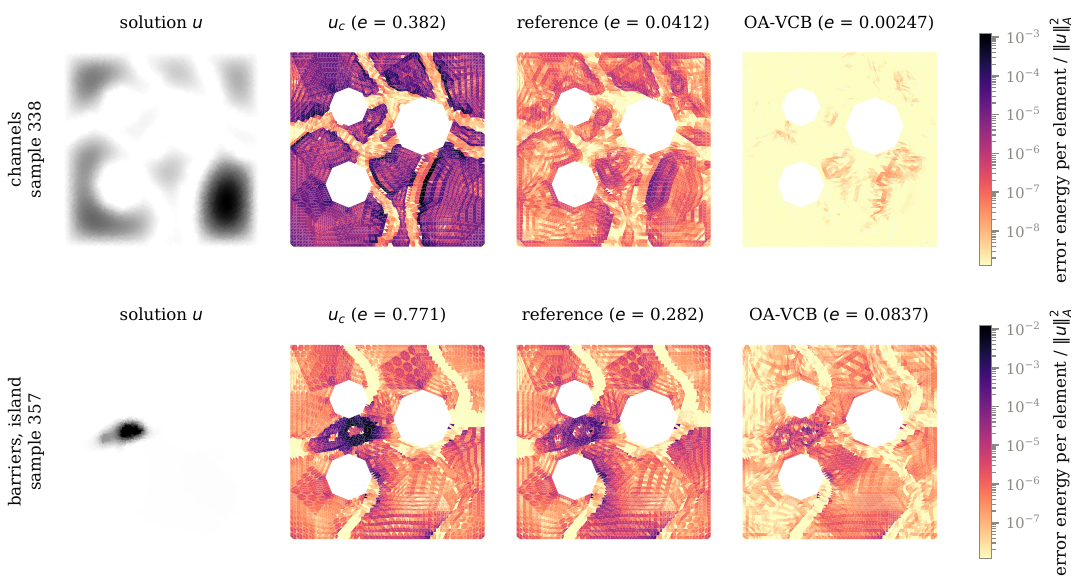}
\caption{Solution and element-wise energy of the error, $a_e|\nabla(u-\hat u)|^2|T_e|/\|u\|_A^2$, for $\hat u = u_c$, the
Chebyshev reference and OA-VCB, on the median channel test sample (by OA-VCB error) and the median island sample of the
barriers (held-out seed 1). Shared logarithmic colour scale per row. On the island sample the remaining OA-VCB error is
concentrated in and around the island.}\label{fig:errormaps}
\end{figure}
On Stage 1's confirmation set the island split is again the dominant feature (Table~\ref{tab:s1sub}, $42$ island and $52$ strict
non-island operators). Against Jacobi-PCG(4) the margin is similar in the two groups ($1.45$ and $1.49$). The other comparators
separate them: the K2-cycle is as accurate as OA-VCB on islands ($0.95$ [$0.80$, $1.13$]) and $1.99$ times less accurate off them;
Gauss--Seidel-preconditioned CG is slightly less accurate on islands ($1.10$) but $2.7$ times more accurate off them ($0.37$); and
the per-instance polynomial candidate is more accurate than OA-VCB on islands ($0.57$ [$0.47$, $0.69$]) and comparable off them
($0.96$). On barriers the learned cycle is thus not uniformly ahead, and its apparent margin depends on which classical control is
chosen and on whether islands are present.

An island indicator is constant on the island, whereas the low-energy mode that it approximates also varies across the
surrounding barrier, so the safeguard does not by itself remove island error. We asked, with a protocol fixed in writing before any
number was computed (\appref{app:harmonic}), how much of OA-VCB's remaining error lies in $A$-harmonic extensions of the
indicators into a patch of two graph rings, added as a sequential $A$-orthogonal correction of the frozen prediction, decided on
$100$ fresh coefficient fields per checkpoint. On the $89$ detector-positive operators the correction reduces the mean energy error
from $0.1219$ to $0.1043$ ($R = 14.4\%$, paired 95\% interval $9.8$--$19.5\%$) and outperforms a patch-indicator control of equal
support by a factor $1.13$, but it does not meet the prespecified support threshold: the outcome is ``little support from the tested
enrichment''. It is a post-processing test on frozen predictions, not an assessment of retraining with enriched deflation.

\subsection{Dependence on the load}\label{sec:loads}\label{sec:rhs}
\begin{table}[tbp]\centering\footnotesize\setlength{\tabcolsep}{4pt}
\caption{Loads other than the training load, against stronger classical smoothers (Stage 1 set LC: 50 fresh channel operators, four off-training load families---$f=-1$, smooth loads of both signs, a localised source and a random-load draw---200 load--operator pairs). Paired geometric-mean ratio of the comparator cycle\textquoteright{}s energy error to the network\textquoteright{}s, with a 95\% bootstrap interval that resamples operators; values below one favour the classical comparator. The fixed-load network was trained with $f=1$ only (initialisation 0); the random-load network was trained on a load distribution. The Chebyshev entry for the random-load network is not comparable with the larger gains of Section~\ref{sec:rhs}, which use other loads and the original test operators.}\label{tab:s1loads}
\begin{tabular}{lllll}\toprule
 & \multicolumn{2}{c}{fixed-load network} & \multicolumn{2}{c}{random-load network} \\ \cmidrule(lr){2-3}\cmidrule(lr){4-5}
comparator & $G$ [95\% CI] & $W$ & $G$ [95\% CI] & $W$ \\\midrule
Chebyshev $\nu=4$ & 0.37 [0.32, 0.43] & 0.18 & 4.56 [4.12, 5.03] & 0.95 \\
Jacobi-PCG(4) & 0.23 [0.19, 0.26] & 0.13 & 2.79 [2.52, 3.08] & 0.94 \\
SGS-PCG(4) & 0.04 [0.04, 0.05] & 0.02 & 0.54 [0.49, 0.59] & 0.09 \\
\bottomrule\end{tabular}\end{table}

The fixed-load networks approximate the map $a\mapsto u(a;f)$ for the training load $f=1$ and their gain does not transfer to other
loads. Positive scaling is exact by Proposition~\ref{prop:homogeneity}, but on Stage 1's four off-training load families
(Table~\ref{tab:s1loads}, $50$ channel operators, $200$ load--operator pairs) the fixed-load network is less accurate than the
Chebyshev reference ($G = 0.37$ [$0.32$, $0.43$]), than the Jacobi-PCG(4) control ($0.23$ [$0.19$, $0.26$]) and than SGS-PCG(4)
($0.04$). Per family and initialisation (\appref{app:stage1}) the fixed-load networks are less accurate than the one-pass
Jacobi-PCG(4) cycle in every case except one, initialisation~2 on the random-load draw ($0.97$ [$0.73$, $1.30$], not
distinguishable from equality). The original study showed the same pattern against the Chebyshev reference
(\appref{app:loads}): with $f=-1$ the fixed-load OA-VCB is less accurate on all $300$ operators ($G=0.21$), and for smooth loads of
both signs, localised sources and random loads the reference is more accurate ($G = 0.20$--$0.60$). The energy safeguard holds
throughout: it prevents damage relative to $u_c$, but it does not promise accuracy.

Training on a load distribution changes the picture. In the original study the random-load networks are $6.2$--$6.9$ times more
accurate than the Chebyshev reference on $f=\pm1$, smooth loads and held-out random loads, at the price of a $1.44$-fold loss on
$f=1$; this six-fold gain is specific to the Chebyshev reference and was not repeated against stronger controls on that set. On Stage
1's set LC the random-load network is $2.79$ [$2.52$, $3.08$] times more accurate than Jacobi-PCG(4) and $4.56$ [$4.12$, $5.03$]
times more accurate than Chebyshev, but $0.54$ [$0.49$, $0.59$] times as accurate as SGS-PCG(4), so a stronger classical
smoother erases the gain on loads even for the distributionally trained network. For localised sources the gain over Jacobi-PCG(4)
is only $1.38$ [$1.06$, $1.76$]. Two operator sets and different load definitions are involved and the figures are not combined.
These tests support broader forcing coverage after distributional training; they do not establish accuracy for arbitrary loads.

\subsection{Design ablations}\label{sec:ablation}\label{sec:transfers}
The following ablations, all from the original study against the Chebyshev reference on its held-out seed~1 (one initialisation
per variant unless stated; \appref{app:ablations}), locate parts of the gain under our training protocol. They are not causal
decompositions: the variants are separately trained networks, and the interaction of components is not additive. The
operator-adapted transfers matter: retraining with geometric prolongation ($m=0$) makes the learned cycle $1.9$ times less
accurate on channels and $4.1$ times on barriers, and it reduces the margin over the correspondingly changed reference from $10.2$
to $6.7$ (channels) and from $2.13$ to $1.48$ (barriers). Retraining with $m=8$ keeps the margin ($11.17$ and $2.01$,
\appref{app:sensitivity}). On barriers, indicator deflation improves the paired accuracy by $G = 2.14$ with plain output and $1.85$
with Galerkin output, and without it no variant beats the reference by more than $15\%$. Learned residual-linear controls, with the
same operator-dependent structure but a smoother map that is linear in its input, still gain on channels ($G = 2.59$ for OA-VCL,
$2.72$ for OA-VCBL), and the nonlinear smoothers add a further $3.36$ and $3.76$; the product $2.72\times3.76 = 10.2$ is algebra
from chaining paired ratios, not an attribution. On barriers the residual-linear models gain little or lose ($1.13$ and $0.70$).
The ablations are consistent with the local diagnostic in that a polynomial-space description does not capture all of the learned
cycle's gain, but they do not identify why.

\subsection{Scope checks}\label{sec:scope}
Separately trained variations of the main configuration are summarised in \appref{app:settings}. Without an externally supplied
coarse solution ($k_0=0$) OA-VCB keeps $G = 10.8$--$11.4$ with $W=1$; a finer hierarchy ($L=4$, $14{,}398$ unknowns) gives
$G=9.72$, a polygonal domain $9.39$, and training and testing on the random-load distribution $6.05$--$6.34$, all against the
Chebyshev reference. The L4-trained model applied unchanged to a finer mesh keeps $G = 8.12$ [$7.50$, $8.67$] at L5 ($58{,}494$
unknowns) and $4.18$ [$3.43$, $4.88$] at L6 ($235{,}774$ unknowns, $21$ of $24$ planned samples, the evaluation having been stopped by a
session limit); neither shows mesh independence, and a model trained at L3 does not keep its margin at L4 ($G=1.95$ on $20$
samples). \emph{Anisotropy.} As a supplementary scope check, we trained and evaluated the construction on the anisotropic
tensor-coefficient family using the original protocol (Table~\ref{tab:aniso}, \appref{app:settings}). The paired geometric-mean energy-error ratios against
the degree-four Chebyshev cycle were $3.50$, $3.51$ and $3.58$ for data seeds $0$, $1$ and $2$, respectively; seed~$0$ was the
development configuration. These comparisons show accuracy gains over this reference in the tested tensor setting. Stronger
classical smoothing budgets, local Krylov-space controls and matched-accuracy solver costs were not evaluated for this family, so
the result does not establish the corresponding advantages or local diagnostic results there. \emph{Elasticity} (plane strain, one
data set and one training run; preliminary): without deflation or tuning OA-VCB is more accurate than its Chebyshev reference on
every test sample ($G=2.32$); tensor diffusion is not a validation of elasticity (\appref{app:settings}). All scope checks use the
reference of the original study; none was re-run against the stronger Stage 1 controls.

\section{Discussion}\label{sec:discussion}

\paragraph{What the controls show.} The classical comparators receive the same transfers, Galerkin levels, coarse solution,
deflation vectors and output projection as OA-VCB, so the margins of Section~\ref{sec:results} measure what four learned update
stages per smoother add to one pass through this cycle, compared with the specific classical smoothers we tested. Part of the
original ten-fold margin was a statement about the Chebyshev reference: a four-step Jacobi-PCG smoothing call already closes a
factor of about $1.6$ on channels, and against SGS-PCG(4) the barrier ordering reverses. The fixed-input diagnostic says something
narrower and cleaner than any complete-cycle ratio: for the same input, the network's output has a smaller energy-norm error than
the exact minimum over the degree-matched space $K_5$ on nearly all channel inputs, whereas the larger space $K_8$ is about as good
or better. It does not say that the learned smoother is optimal, that no other classical smoother matches it, or that its gain
depends on nonlinearity.

\paragraph{Why local gains do not make a solver.} A smoother that beats a degree-matched space for its own input is a local object.
In a complete cycle, later inputs change, and a larger classical smoothing budget is cheap on the operator-adapted hierarchy: eight
Jacobi-PCG steps per call already overtake OA-VCB on barriers and fourteen on both families, with substantially less counted
arithmetic (Tables~\ref{tab:s1brackets} and~\ref{tab:s1cost}). Against solvers that converge to the same accuracy the measured gap
is one to two orders of magnitude in time and grew with the mesh, and the network inference alone exceeded a complete two-grid PCG
solve including its set-up. Reuse of an operator across many right-hand sides, batched GPU execution and optimised implementations
were not measured; they would have to be compared with equally optimised classical code, and a correction that is nonlinear in the
load would need a suitable outer method inside an iteration. We draw no solver conclusion.

\paragraph{Loads, islands and complexity.} The network is positively homogeneous but not linear in the residual. Trained with one
load, it learns corrections for the residuals that load produces, and its advantage disappears on other loads; distributional
training restores a margin that SGS-PCG(4) still removes (Section~\ref{sec:loads}). Floating islands retain substantial error
after indicator deflation, and the tested sequential harmonic enrichment removes only a modest part of it. The operator complexity
of the hierarchy, $4.5$ for three levels and $8.6$ for four, is a measured limitation for deeper hierarchies or three dimensions;
truncated prolongations \cite{vanek1996} or richer coarse spaces \cite{galvis2010,spillane2014geneo,chung2018cem,ma2022,rudikov2025lsino}
were not tested.

\paragraph{What the study offers.} The need for strong baselines in learned PDE solvers is established \cite{mcgreivy2024}; what this
study adds is specific evidence in one explicit construction. For readers developing learned smoothers it provides a cycle in which
the classical comparators share every non-learned ingredient, a fixed-input Krylov-space diagnostic that separates local
approximation from full-cycle accuracy, the classical budgets at which the paired accuracy crosses, adverse cases (barriers with
islands, off-training loads) reported separately rather than averaged away, and the per-operator records and scripts needed to
re-examine every comparison. In this study the hierarchy and the learning also interacted: with geometric transfers the learned
cycle was $1.9$ and $4.1$ times less accurate and its margin over the classical reference smaller.

\paragraph{Scope and threats to validity.} All experiments are two-dimensional with $P_1$ elements, with contrast $10^4$ for the
isotropic families, $3{,}486$ unknowns in the main configuration, and one family of perforated domains plus one polygon. The learned
cycle runs in single precision and the classical cycles in double precision, without a precision ablation. Stage 1 reuses stored
checkpoints, three initialisations on the same operators, so intervals describe variability over test operators only. The
frozen-scale gate of Stage 1 was not executed at run time (Section~\ref{sec:evidence}). Timings are Python on one CPU with oracle
stopping rules. Anisotropy and elasticity were not evaluated against the stronger controls.

\section{Conclusion}\label{sec:conclusion}
We asked what learned smoothing adds to one pass of a multigrid cycle whose transfers, coarse operators, coarsest solve and
deflation vectors are computed from each operator. Against four-step Jacobi-PCG smoothing in the same cycle the learned pass is about
six times more accurate in the energy norm on conductive channels and about $1.4$ times on insulating barriers, and for a fixed input
the learned smoother improves on the energy minimum over the degree-matched Krylov space on nearly all channel inputs. These gains
do not carry over to larger classical budgets: a fourteen-step Jacobi-PCG cycle is the more accurate on both families, and the
implementation is slower than converged classical solvers at matched accuracy. The advantage also depends on the training load,
and floating islands keep substantial error. Learned smoothing is therefore a measurable source of local and one-pass accuracy
within a prescribed multilevel cycle, not a practical solver in the tested implementation. Whether such local gains can lower the
total cost of reaching a tolerance remains open.

\section*{CRediT authorship contribution statement}
\textbf{Kumar Kaushik Ranjan:} Conceptualization, Methodology, Software, Validation, Formal analysis, Investigation, Data
curation, Writing -- original draft, Writing -- review \& editing, Visualization.

\section*{Declaration of competing interest}
The author declares no known competing financial interests or personal relationships that could have appeared to
influence the work reported in this paper.

\section*{Funding}
This research did not receive any specific grant from funding agencies in the public, commercial, or not-for-profit sectors.

\section*{Declaration of generative AI and AI-assisted technologies in the manuscript preparation process}
During the preparation of this work the author used Claude (Anthropic) in order to write and run software for the numerical
experiments, to assemble the tables and figures from stored results and to draft and edit the manuscript text, and ChatGPT
(OpenAI) in order to review the study protocol, the results and the manuscript critically. The use of
these tools in the research itself is described in Section~\ref{sec:costmeasure}. After using these tools, the author reviewed and
edited the content as needed and takes full responsibility for the content of the published article.

\section*{Code and data availability}
The source code (data generation, models, training and evaluation, identified by the hash \texttt{f8653612eb84} stored with every
result), the Stage 1 protocol, amendments, reports and per-operator records, the trained checkpoints and the scripts that regenerate
the tables and figures from the records are publicly available at \url{https://github.com/kkr12334/OFW_NO} and archived on Zenodo as
version 1.0.1, \href{https://doi.org/10.5281/zenodo.23477955}{doi:10.5281/zenodo.23477955} \cite{ranjan2026release} (concept DOI for all
versions: 10.5281/zenodo.23477954). Code is released under the MIT licence and data and records under CC BY 4.0. The archived
release (854 files) carries a SHA-256 list of every file (release digest
\texttt{4370b4e4ec64989c\allowbreak 8f26b75096267ec5\allowbreak 7307f55f3a902c19\allowbreak 0109ea53258d5670}). The final versions of the Stage 1 table and figure
generators used for this manuscript are included with its source files.

\clearpage
\appendix
\section{Contained classical cycles, implementation details and numerical checks}\label{app:impl}
\begin{table}[tbp]\centering\footnotesize\setlength{\tabcolsep}{4pt}
\caption{Classical cycles contained in the architecture (Proposition~\ref{prop:containment}, same hierarchy, coarse solution, deflation and Galerkin output; $\nu=4$): fixed-interval Chebyshev smoothing ($\lambda=3$, same weights on every level) and weighted Jacobi ($\omega=2/3$), together with the adaptive Chebyshev reference used in the other tables. The contained cycles are evaluated in double precision with the unregularised Galerkin solves of the mathematical construction; the regularisation and finite-precision safeguards of the deployed network are described in Section~\ref{sec:properties}. (a) Mean energy errors and paired $G$ of each contained cycle over the reference. (b) Paired $G$ of each classical cycle over OA-VCB ([95\% CI] for the contained cycles) and OA-VCB's win fraction $W$. Elasticity: no deflation; one data set.}\label{tab:contained}
{\small (a) Mean energy errors}\par\smallskip
\begin{adjustbox}{max width=\linewidth}
\begin{tabular}{lllllll}\toprule
set & OA-VCB & reference & fixed Cheb. & Jacobi & fixed Cheb./ref.\ $G$ & Jacobi/ref.\ $G$ \\\midrule
channels, seed 1 & 2.55e-3 & 2.59e-2 & 2.58e-2 & 4.03e-2 & 0.995 & 1.57 \\
channels, seed 2 & 2.79e-3 & 2.50e-2 & 2.49e-2 & 3.85e-2 & 0.995 & 1.55 \\
channels, seed 3 & 2.92e-3 & 2.63e-2 & 2.62e-2 & 4.05e-2 & 0.995 & 1.54 \\
barriers, seed 1 & 4.36e-2 & 7.51e-2 & 7.62e-2 & 8.63e-2 & 1.052 & 1.51 \\
barriers, seed 2 & 6.96e-2 & 8.51e-2 & 8.61e-2 & 9.64e-2 & 1.040 & 1.45 \\
barriers, seed 3 & 6.07e-2 & 8.33e-2 & 8.50e-2 & 9.50e-2 & 1.056 & 1.50 \\
$k_0=0$, seed 1 & 7.62e-3 & 8.52e-2 & 8.84e-2 & 1.15e-1 & 1.039 & 1.35 \\
$k_0=0$, seed 2 & 7.58e-3 & 8.04e-2 & 8.36e-2 & 1.09e-1 & 1.041 & 1.36 \\
$k_0=0$, seed 3 & 8.13e-3 & 8.55e-2 & 8.85e-2 & 1.15e-1 & 1.036 & 1.35 \\
L4, $k_0=3$, seed 1 & 2.21e-3 & 2.01e-2 & 2.02e-2 & 3.08e-2 & 1.008 & 1.54 \\
elasticity (prelim.) & 2.47e-2 & 5.50e-2 & 5.38e-2 & 9.84e-2 & 0.975 & 1.82 \\
\bottomrule\end{tabular}
\end{adjustbox}
\par\medskip
{\small (b) Gains of OA-VCB}\par\smallskip
\begin{adjustbox}{max width=\linewidth}
\begin{tabular}{llllll}\toprule
set & ref./OA-VCB $G$ & fixed Cheb./OA-VCB $G$ & $W$ & Jacobi/OA-VCB $G$ & $W$ \\\midrule
channels, seed 1 & 10.23 & 10.19 [9.69, 10.66] & 1.00 & 16.02 [15.27, 16.73] & 1.00 \\
channels, seed 2 & 9.76 & 9.71 [9.12, 10.27] & 1.00 & 15.11 [14.16, 16.00] & 1.00 \\
channels, seed 3 & 9.68 & 9.63 [9.07, 10.19] & 1.00 & 14.95 [14.08, 15.79] & 1.00 \\
barriers, seed 1 & 2.13 & 2.24 [1.99, 2.51] & 0.93 & 3.22 [2.83, 3.64] & 0.94 \\
barriers, seed 2 & 1.93 & 2.01 [1.76, 2.28] & 0.83 & 2.80 [2.39, 3.27] & 0.87 \\
barriers, seed 3 & 1.93 & 2.04 [1.78, 2.32] & 0.90 & 2.90 [2.53, 3.33] & 0.93 \\
$k_0=0$, seed 1 & 11.39 & 11.83 [11.38, 12.30] & 1.00 & 15.34 [14.74, 15.93] & 1.00 \\
$k_0=0$, seed 2 & 10.91 & 11.36 [10.78, 11.91] & 1.00 & 14.80 [14.05, 15.52] & 1.00 \\
$k_0=0$, seed 3 & 10.82 & 11.21 [10.77, 11.67] & 1.00 & 14.57 [14.01, 15.17] & 1.00 \\
L4, $k_0=3$, seed 1 & 9.72 & 9.80 [9.13, 10.42] & 1.00 & 14.98 [13.93, 15.92] & 1.00 \\
elasticity (prelim.) & 2.32 & 2.26 [2.15, 2.37] & 1.00 & 4.22 [3.93, 4.53] & 1.00 \\
\bottomrule\end{tabular}
\end{adjustbox}
\end{table}

\paragraph{Contained classical cycles.} Table~\ref{tab:contained} evaluates the two contained members of
Proposition~\ref{prop:containment} on the main test sets. The contained fixed-interval Chebyshev member is as accurate as the
adaptive reference to within a few per cent (paired ratio $0.975$--$1.056$), so OA-VCB's margins over it are essentially those
over the reference ($G = 9.6$--$10.2$ on channels, $2.0$--$2.2$ on barriers, $2.26$ for elasticity). The contained Jacobi member
is less accurate on every set ($96$--$100\%$ of the samples on the scalar sets).

\paragraph{Numerical check of the containment.} With the weights of the construction in Proposition~\ref{prop:containment},
the network and the classical cycle agree to $3.6\times10^{-10}$ for Jacobi weights (diffusion, \path{mathcheck/check_oa_vcp.log};
$1.1\times10^{-10}$ tensor coefficient, \path{mathcheck/check_oa_vcb_tensor.log}; $1.4\times10^{-10}$ elasticity,
\path{mathcheck/check_oa_vcb_elasticity.log}; $\omega = 2/3$, $\nu = 4$), to $4.8\times10^{-10}$ for the Chebyshev weights of (b)
with $\lambda = 3$ and $q = 10$, and to $3.7\times10^{-10}$ for the first construction with a random cubic $p$ whose residual
polynomial has a pair of complex roots (\path{mathcheck/check_poly_containment.log}; maximum relative differences of the
outputs on twelve barrier test samples). On every level operator of these samples the product form of (b) and the three-term
recurrence of the Chebyshev semi-iteration give the same smoother vectors to a relative difference of $5.9\times10^{-16}$.

\paragraph{Transposed transfers.} With $T = [\Phi_k^\star\;E_F]$, \eqref{eq:inverse} reads
$A_k^{-1} = T\,\operatorname{diag}(S_k,A_{FF})^{-1}\,T^\top$: the residual is analysed in the \emph{dual} coordinates $T^\top r$
and the correction synthesised with $T$. Any invertible synthesis $T$ admits this form. What it does not admit is a
\emph{block-diagonal} latent map between the primal pair $T^{-1}$ (analysis) and $T$ (synthesis), as used by exactly
invertible lifting networks: there the latent map would have to equal $T^{-1}A_k^{-1}T = \operatorname{diag}(S_k,A_{FF})^{-1}\,T^\top T$,
which is block diagonal only if $T^\top T$ is (\path{mathcheck/check_oa_transpose_pair.log}). With $X^{(m)}$ replaced by the
exact $-A_{FF}^{-1}A_{FC}$ and exact level solves the cycle reproduces $A^{-1}$ to $10^{-12}$--$10^{-14}$; with the approximation
the coupling between levels that $X^{(m)}$ leaves is handled by the multiplicative structure of the cycle.

\paragraph{Reproduction of the kernel (Proposition~\ref{prop:constants}).} On level $k$ the statement covers the fine nodes
whose $m$-step neighbourhood in the level-$k$ (Galerkin) graph satisfies both hypotheses. On the finest level that
neighbourhood spans $m$ mesh widths, and the measured rigid-body defect of the elasticity transfers is at round-off for most
rows (median $7\times10^{-17}$), but not for all (maximum $0.40$ over rows not directly coupled to $\Gamma_D$, whose $m$-step
neighbourhoods may still reach it). On coarser levels the Galerkin stencils widen, the hypotheses cover fewer rows (41\% on the
second level, none on the third) and the defects grow (median over all rows $1.5\times10^{-4}$ and $1.6\times10^{-2}$;
\path{mathcheck/check_oa_vcb_elasticity.log}). We report these measured defects rather than a physical-width bound.

\paragraph{Homogeneity in finite precision.} In the single-precision implementation $\sigma$ is evaluated as a scaled root
mean square ($\max|z|\cdot\|z/\max|z|\|_{\mathrm{rms}}$, which cannot underflow), and $\mathcal{N}(\lambda r) = \lambda\mathcal{N}(r)$
holds to $7\times10^{-7}$ for residual scales from $10^{-20}$ to $10^{8}$ on CPU and GPU (\path{mathcheck/check_projection_precision.log}).
At $10^{-30}$ the smallest output components fall below the smallest normal single-precision number; GPU kernels flush such
values to zero and the deviation grows to $3\times10^{-5}$ (finiteness and the energy bound are unaffected).

\paragraph{The energy safeguard in the implementation.} The implementation differs from Proposition~\ref{prop:safeguard} in
three ways; each preserves energy non-increase, none preserves the minimisation equality in general. (i)~\emph{Data.}
Deflation and the output projection use the residual and the fine-level operator in double precision, as assembled; the
cycle itself runs in single precision, which changes the directions $W$ but not the argument. (ii)~\emph{Regularisation.} The
Jacobi-scaled Gram matrix $G$ is symmetrised and a ridge $R = 10^{-10}I$ is added; for $(G+R)y = g$ the energy change is
$-y^\top Gy - 2y^\top Ry\le 0$, but $y$ is no longer the minimiser over $\operatorname{span}W$ when $R$ matters (nearly
dependent directions). (iii)~\emph{Acceptance.} For the computed correction $d$ of a residual $b$ (the deflation step and the
output projection alike), energy non-increase is equivalent to $2d^\top b - d^\top A d\ge 0$; this is evaluated with the
double-precision operator, and if it fails $d$ is replaced by $d\,\max(d^\top b,0)/d^\top Ad$, which may shrink or reject it.
The implemented guarantee is therefore $\|e-c\|_A\le\|e - x_0\|_A\le\|e\|_A$ up to double-precision rounding in the acceptance
test, with $c$ not necessarily the minimiser. No eigendecomposition is used in the forward pass, so the gradients stay defined
at repeated eigenvalues and rank changes; zero right-hand sides are mapped to zero explicitly. Numerically
(\path{mathcheck/check_projection_precision.log}): no violation at tolerance $10^{-12}$ on real samples with random and
classical weights, on 210 stress cases with nearly or exactly duplicate directions (as projection and as deflation with zero
padding), or for residual scales from $10^{-30}$ to $10^{8}$; finite gradients at duplicate columns and at zero and tiny
residuals; the acceptance step never triggered. An earlier implementation formed $G$ in single precision and broke the bound in
11 of the 210 stress cases. Re-evaluating trained checkpoints with the present implementation changed the corrections by at most
$7\times10^{-6}$ relative and the mean errors by less than $10^{-5}$ relative, with no acceptance failures and no rank loss
(\path{diagnostics_cost/checkpoint_audit_*.log}).

\section{Proofs}\label{app:proofs}
\paragraph{Proof of Lemma~\ref{lem:routev}.} $E_F^\top A_k\Phi_k^\star = A_{FC} - A_{FF}A_{FF}^{-1}A_{FC} = 0$. The matrix
$T = [\Phi_k^\star\;E_F]$ is unit block lower triangular, hence invertible, and $T^\top A_k T$ is block diagonal; inverting
$A_k = T^{-\top}\operatorname{diag}(S_k,A_{FF})T^{-1}$ gives \eqref{eq:inverse}.\hfill$\square$

\paragraph{Reproduction of the kernel.}
\begin{proposition}[local reproduction of the kernel]\label{prop:constants}
Let $t$ be a level-$k$ vector with restriction $t_C$ to the coarse nodes (constants for diffusion; translations and rotations for
elasticity, which $P_1$ interpolation reproduces), and let $\mathcal{N}_r(f)$ denote the nodes within $r$ graph steps of a fine node
$f$. If $P_k t_C = t$ on $\mathcal{N}_m(f)$ and $(A_k t)_g = 0$ for every fine node $g\in\mathcal{N}_{m-1}(f)$, then
$(\widehat\Phi_k t_C)_f = t_f$.
\end{proposition}
\begin{proof}
Write $\widehat\Phi^{(j)} = [I;X^{(j)}]$; the update in \eqref{eq:jacobi} applied to $t_C$ at a fine node $g$ is
$-\omega\,(A_k\widehat\Phi^{(j)}t_C)_g/A_{gg}$, which depends on $\widehat\Phi^{(j)}t_C$ on $\mathcal{N}_1(g)$ only. We show by
induction that $\widehat\Phi^{(j)}t_C = t$ on $\mathcal{N}_{m-j}(f)$. For $j = 0$ this is the hypothesis. If it holds for $j$ and
$g\in\mathcal{N}_{m-j-1}(f)$, then $\mathcal{N}_1(g)\subset\mathcal{N}_{m-j}(f)$, so the update at $g$ is
$-\omega(A_kt)_g/A_{gg} = 0$ and $\widehat\Phi^{(j+1)}t_C = t$ on $\mathcal{N}_{m-j-1}(f)$. Take $j = m$.
\end{proof}
The statement is graph-local and conditional; measured reproduction defects of the elasticity transfers are in
\appref{app:impl}.

\paragraph{Containment (Proposition~\ref{prop:containment}): explicit construction.} Set the encoder, step maps and decoder of both
smoothers as follows: channel $1$ of the latent carries $z/\sigma$, channel $0$ starts at zero, $V_t$ selects channel $0$,
$\mathcal{L}_t$ maps $(s,\,q)$ to $a_{\nu-t}\,s_1 + q_0$ in channel $0$ and to zero elsewhere, $\mathcal{M}_t = 0$, and the decoder
reads channel $0$.
\begin{proof}
Write $s_0 = x/\sigma$ for channel $0$ at node $i$; its aggregate is
$q_0 = \sum_{j\ne i}W_{ij}x_j/\sigma = \big((B_k - I)x\big)_i/\sigma$. Since the step map is added to the current state, step $t$
gives $s_0 + a_{\nu-t}s_1 + q_0 = \big(B_kx + a_{\nu-t}D_k^{-1}b\big)_i/\sigma$, that is, $x_t = B_kx_{t-1} + a_{\nu-t}D_k^{-1}b$ with
$x_0 = 0$, and Horner's rule gives $x_\nu = \sum_{j=0}^{\nu-1}a_jB_k^{\,j}D_k^{-1}b = p(B_k)D_k^{-1}b$. Multiplying the decoded
channel by $\sigma$ undoes the scaling, and $A_k^{-1}b - p(B_k)B_kA_k^{-1}b = R(B_k)A_k^{-1}b$. Conversely, a real $R$ of degree at
most $\nu$ with $R(0) = 1$ has the form $1-\mu p(\mu)$ with $p$ real of degree at most $\nu-1$. The remaining operations in
\eqref{eq:cycle}--\eqref{eq:projection} are the same in the network and the classical cycle.
\end{proof}
\emph{Special cases.} When $R$ has real roots $1/\omega_t$, the step maps $\mathcal{L}_t(s,q) = \omega_t(s_1 - s_0 - q_0)$ in channel
$0$ realise the Richardson steps $x \leftarrow x + \omega_t D_k^{-1}(b - A_k x)$, $t = 1,\dots,\nu$, $x = 0$ initially, for which
$R(\mu) = \prod_{t=1}^{\nu}(1 - \omega_t\mu)$: step $t$ gives $s_0 + \omega_t(s_1 - s_0 - q_0) = \big(x + \omega_t
D_k^{-1}(b - A_k x)\big)_i/\sigma$ and multiplies the error by $I-\omega_tB_k$. (a)~$\omega_t=\omega$ is weighted Jacobi smoothing.
(b)~Degree-$\nu$ Chebyshev smoothing on $[\lambda/q,\lambda]$ takes $\omega_t^{-1} = \theta - \delta\cos\big((2t-1)\pi/(2\nu)\big)$ with
$\theta = (\lambda + \lambda/q)/2$, $\delta = (\lambda - \lambda/q)/2$; both $\prod_t(1-\omega_t\mu)$ and
$T_\nu\big((\theta-\mu)/\delta\big)/T_\nu(\theta/\delta)$ are polynomials of degree $\nu$ in $\mu$ that equal $1$ at $\mu = 0$ and
vanish at the $\nu$ distinct Chebyshev nodes $\mu_t = \omega_t^{-1}$, so they coincide \cite{saad2003}.

\paragraph{Positive homogeneity (Proposition~\ref{prop:homogeneity}).}
\begin{proof}
For a smoother, $z$ and $\sigma$ scale with $\lambda$, the features $z/\sigma$ and $\log d$ do not, and the output carries the
factor $\sigma$; so $\mathcal{S}(\lambda b) = \lambda\mathcal{S}(b)$ (and $\mathcal{S}(0) = 0$ by definition). The residual
updates, transfers, coarsest solve and deflation \eqref{eq:deflation} are linear. In \eqref{eq:projection} every column of $W$
except $Z$ scales with $\lambda$, so $\operatorname{span}W$ is unchanged and the projection of $\lambda(r - Ax_0)$ onto it is
$\lambda$ times the projection of $r - Ax_0$.
\end{proof}

\paragraph{Energy-norm safeguard (Proposition~\ref{prop:safeguard}).}
\begin{proof}
$W^\top A W y = W^\top(r - Ax_0) = W^\top A(e - x_0)$ are the normal equations of $\min_y\|e - x_0 - Wy\|_A$, and $y = 0$ is
admissible. Likewise \eqref{eq:deflation} is the $A$-orthogonal projection of $e$ onto $\operatorname{span}Z$, with $w = 0$
admissible.
\end{proof}

\section{Stage 1: protocol chronology, provenance and further results}\label{app:stage1}
\paragraph{Chronology.} The Stage 1 protocol (16 October) fixed the comparators, the primary comparisons and the statistics before any
confirmation operator was generated; amendment~1 (16 October) and a pre-confirmation note (17 October) recorded fitted
coefficients, checkpoint hashes and script hashes before the confirmation numbers were computed; the confirmation run was
executed once. Amendment~2 (18 October) fixed endpoints R1--R4 (channel bracket Jacobi-PCG(8)--PCG(14); barrier bracket
PCG(5)--PCG(8); PCG(14) against SGS-PCG(4), descriptive; local $K_5$ criterion $\ge90$ of $100$ inputs with $\rho\le0.99$), the repaired
timings and the operation ledger, before the reserve operators were generated (random seeds $2026101603$ for channels
and $2026101604$ for barriers; confirmation set: $2026101601$ and $2026101602$). SHA-256 lists of the frozen protocol, both amendments,
the pre-confirmation note and the reports are stored with them in \path{ext/}.

\paragraph{Departure from the protocol: the frozen-scale gate.} Amendment~2 required a check, at run time, of the frozen scale file
(gate B4). This check was not executed when the runs were made. The SHA-256 of the scale file
($8d0cc6f1\ldots$) matches the hash recorded before the runs, but this comparison was made retrospectively; it does not prove the
content of the file at run time, and we do not claim that the gate was passed.

\paragraph{Operation ledger.} The ledger counts multiply--accumulate equivalents (one multiply and one add; a diagonal scaling,
division or vector update of length $n$ counts $n$; a dot product counts $n$ and one scalar global reduction) of one pass of the
complete cycle, set-up excluded. It excludes the operator-adapted hierarchy and $u_c$, which are common to all methods, and
method-specific set-up: the Chebyshev eigenvalue bounds, for example, cost a further $2.59$ million MAC-equivalents on channels,
which would lower the OA-VCB/Chebyshev ratio from $414$ to $165$. The ledger also omits the small dense solves of the Galerkin
wrapper, does not enumerate every Gram symmetrisation, scaling and masking operation, and counts the unpadded wrapper: on
channels, where no near-kernel vector is detected, the executed network wrapper still carries a padded zero column, whereas the
ledger uses the wrapper without it. The dense-layer and cycle sparse-product counts agree with independent counters, which does
not certify every executed operation. Activation evaluations ($13.1$ million per OA-VCB pass) and nonlinear global reductions
($18$) are reported separately. The ledger is neither a time nor a communication volume, and its ratios (for example about $161$
between OA-VCB and the $14$-step Jacobi-PCG cycle) hold only under this convention.

\paragraph{Timings.} One Intel Xeon processor at $2.10$\,GHz with two threads, Python~3.13, NumPy~2.5, PyTorch~2.14 (CPU) and PyAMG~5.3;
each component was timed five times per operator, and tables report medians over $50$ operators of the per-operator medians. ``To
OA-VCB error'' selects, offline and per operator, the first iteration whose energy error reaches the one-pass OA-VCB error (an
oracle rule). ``Estimated end-to-end'' is the median over operators of the per-operator sum of that time and the operator's own
$u_c$ construction time, not the sum of two population medians. The operator-data share of the OA-VCB pass (about $87$ and
$86$\,ms) comes from a separate per-component profile taken on the last of the five repetitions; it and the median pass time are
different summaries and do not partition one measurement. The direct solve is never charged $u_c$.

\paragraph{Independent recomputation.} The Stage 1 statistics (paired geometric means, win fractions, bootstrap endpoints with the
fixed random seed $20261016$ and $10{,}000$ resamples) were recomputed from the per-operator errors by an independent script that does
not use the study's statistics code; agreement is to $9\times10^{-16}$. A second independent audit of the principal comparisons
reported here (confirmation and reserve comparisons, pooled load comparisons with operator-cluster resampling, per-instance
candidates, local counts and medians, matched-error timings and ledger ratios) agreed with the stored values to $5.3\times10^{-15}$. The
table generator also re-computes every interval and compares it with the stored endpoints to $10^{-9}$.

\paragraph{Further results.} Table~\ref{tab:s1ladder} lists every one-pass comparator of Stage 1 on the confirmation set;
Table~\ref{tab:s1sub} splits the barrier operators by floating islands; Table~\ref{tab:s1loadfam} gives the off-training loads by
family.
\begin{table}[tbp]\centering\footnotesize\setlength{\tabcolsep}{3pt}
\caption{All one-pass comparators of Stage 1 on set C (100 operators per family; OA-VCB initialisation 0), in the same cycle with only the smoothers replaced. $G(\mathrm{Cheb}\!\to\!X)$: Chebyshev reference error divided by the comparator\textquoteright{}s; $G(X\!\to\!\mathrm{OA\text{-}VCB})$: comparator error divided by OA-VCB\textquoteright{}s, with the pointwise 95\% bootstrap interval; $W$: fraction of operators on which OA-VCB is more accurate. Steepest descent and minimal residual use four Jacobi-preconditioned steps per call; the K2-cycle is a recursive flexible-CG K-cycle on the coarse levels with Chebyshev smoothers; the Lanczos-bound Chebyshev variant uses an interval ratio chosen on T; the fitted polynomials are fitted on T, shared by all levels (contained class) or one set per level (not contained); symmetric Gauss--Seidel uses two symmetric sweeps. The per-instance candidate\textquoteright{}s error is an upper bound on that of the best polynomial class member, not a certified optimum. OA-VCL and OA-VCBL are learned models of the original study with residual-linear smoothers (the fitted Galerkin output weights of OA-VCBL need not make its whole input--output map linear).}\label{tab:s1ladder}
\begin{tabular}{>{\raggedright\arraybackslash}p{5.8cm}lll}\toprule
\multicolumn{4}{l}{\emph{channels}} \\
comparator $X$ & $G(\mathrm{Cheb}\!\to\!X)$ & $G(X\!\to\!\mathrm{OA\text{-}VCB})$ & $W$ \\\midrule
steepest descent, 4 steps (Jacobi) & 0.82 & 11.26 [10.35, 12.16] & 1.00 \\
minimal residual, 4 steps (Jacobi) & 0.83 & 11.06 [10.15, 11.96] & 1.00 \\
K2-cycle (Chebyshev smoothers) & 0.97 & 9.48 [8.71, 10.25] & 1.00 \\
Chebyshev, Lanczos bounds & 1.31 & 6.99 [6.41, 7.56] & 0.99 \\
fitted polynomial, shared (contained) & 1.79 & 5.14 [4.79, 5.51] & 1.00 \\
fitted polynomials, per level & 1.86 & 4.93 [4.54, 5.30] & 0.99 \\
Jacobi-PCG(4) (primary) & 1.61 & 5.70 [5.25, 6.13] & 0.99 \\
per-instance polynomial candidate & 2.22 & 4.14 [3.82, 4.45] & 0.99 \\
symmetric Gauss--Seidel, two sweeps & 1.27 & 7.23 [6.63, 7.83] & 0.99 \\
SGS-PCG(4) & 7.94 & 1.16 [1.08, 1.23] & 0.79 \\
OA-VCL (residual-linear smoother) & 2.48 & 3.70 [3.41, 4.04] & 1.00 \\
OA-VCBL (residual-linear smoother) & 2.60 & 3.53 [3.28, 3.80] & 1.00 \\
\midrule
\multicolumn{4}{l}{\emph{barriers}} \\
comparator $X$ & $G(\mathrm{Cheb}\!\to\!X)$ & $G(X\!\to\!\mathrm{OA\text{-}VCB})$ & $W$ \\\midrule
steepest descent, 4 steps (Jacobi) & 0.82 & 2.42 [2.19, 2.67] & 0.92 \\
minimal residual, 4 steps (Jacobi) & 0.79 & 2.52 [2.27, 2.78] & 0.91 \\
K2-cycle (Chebyshev smoothers) & 1.40 & 1.42 [1.26, 1.60] & 0.75 \\
Chebyshev, Lanczos bounds & 1.12 & 1.79 [1.63, 1.96] & 0.91 \\
fitted polynomial, shared (contained) & 1.13 & 1.76 [1.60, 1.93] & 0.91 \\
fitted polynomials, per level & 1.30 & 1.54 [1.40, 1.68] & 0.84 \\
Jacobi-PCG(4) (primary) & 1.37 & 1.45 [1.33, 1.59] & 0.82 \\
per-instance polynomial candidate & 2.65 & 0.75 [0.68, 0.83] & 0.34 \\
symmetric Gauss--Seidel, two sweeps & 1.01 & 1.97 [1.79, 2.17] & 0.91 \\
SGS-PCG(4) & 3.44 & 0.58 [0.50, 0.68] & 0.29 \\
OA-VCL (residual-linear smoother) & 0.67 & 2.99 [2.72, 3.30] & 0.97 \\
OA-VCBL (residual-linear smoother) & 1.10 & 1.81 [1.65, 1.99] & 0.88 \\
\bottomrule\end{tabular}\end{table}

\begin{table}[tbp]\centering\footnotesize\setlength{\tabcolsep}{4pt}
\caption{Barrier operators of set C split by floating islands (geometric label; 42 island operators and 52 strict non-island operators of 100; the remaining six are excluded here as ambiguous or too few). $G(X\!\to\!\mathrm{OA\text{-}VCB})$ for OA-VCB initialisation 0 with pointwise 95\% bootstrap intervals; values below one favour the comparator. Statements about intervals are not tests of equivalence.}\label{tab:s1sub}
\begin{tabular}{lll}\toprule
comparator & island & strict non-island \\\midrule
Jacobi-PCG(4) & 1.45 [1.20, 1.72] & 1.49 [1.37, 1.63] \\
K2-cycle (Chebyshev) & 0.95 [0.80, 1.13] & 1.99 [1.82, 2.17] \\
SGS-PCG(4) & 1.10 [0.90, 1.33] & 0.37 [0.32, 0.43] \\
per-instance polynomial candidate & 0.57 [0.47, 0.69] & 0.96 [0.89, 1.04] \\
\bottomrule\end{tabular}\end{table}

\begin{table}[tbp]\centering\footnotesize\setlength{\tabcolsep}{3pt}
\caption{Stage 1 set LC by load family (50 channel operators each): $G(\text{Jacobi-PCG(4)}\!\to\!\text{network})$ with the pointwise 95\% bootstrap interval, for the three fixed-load initialisations and the random-load network. Values below one mean that the one-pass Jacobi-PCG(4) cycle is more accurate.}\label{tab:s1loadfam}
\begin{tabular}{lllll}\toprule
 & \multicolumn{3}{c}{fixed-load network, initialisation} & random-load \\ \cmidrule(lr){2-4}
load family & 0 & 1 & 2 & network \\\midrule
$f=-1$ & 0.15 [0.13, 0.18] & 0.09 [0.08, 0.09] & 0.48 [0.47, 0.50] & 3.49 [3.10, 3.89] \\
smooth, both signs & 0.17 [0.13, 0.23] & 0.09 [0.07, 0.12] & 0.54 [0.46, 0.64] & 3.64 [3.29, 3.99] \\
localised source & 0.27 [0.19, 0.39] & 0.61 [0.49, 0.77] & 0.56 [0.42, 0.74] & 1.38 [1.06, 1.76] \\
random-load draw & 0.35 [0.22, 0.58] & 0.26 [0.16, 0.46] & 0.97 [0.73, 1.30] & 3.47 [3.10, 3.84] \\
\bottomrule\end{tabular}\end{table}

\paragraph{Threats to validity.} (i)~The Stage 1 results reported here evaluate stored data-seed-1 checkpoints of the original study
(three fixed-load OA-VCB initialisations, residual-linear controls and the random-load model of the load challenge); no neural network
was trained or retrained for them. (ii)~Three initialisations on the same $100$ operators are not $300$ independent operators.
(iii)~Bootstrap intervals are pointwise and conditional on the trained models. (iv)~The learned cycle runs in single precision
inside double-precision deflation and output, whereas the classical cycles run in double precision; no precision ablation was run.
(v)~The per-instance polynomial candidate's error is an upper bound on the error of the best class member, not a certified optimum,
and a candidate that loses does not establish superiority of OA-VCB over the class optimum. (vi)~$K_5$ is a degree control, not an equal-information or equal-cost control. (vii)~The two PCG-type families
used as smoothing calls are not the whole-system solvers. (viii)~Anisotropy and elasticity were not evaluated against the Stage 1
controls. (ix)~The frozen-scale gate was not executed at run time (above).

\paragraph{Provenance.} Stage 1 records are the per-operator JSON files under \path{ext/runs/}; \path{make_tables_stage1.py} reads
them, writes \path{tables/numbers_stage1.txt} (every Stage 1 number quoted in the text) and \path{tables/MANIFEST_stage1.txt}
(SHA-256 of every input file). The model code hash of all records is \texttt{f8653612eb84} (\path{run_identity.py}).

\section{Original-study main comparison, larger classical budgets and platform costs}\label{app:budget}
\begin{table}[tbp]\centering\footnotesize\setlength{\tabcolsep}{4pt}
\caption{Main comparison of the original study, held-out data seeds 1--3 (100 test operators each; L3, $k_0=2$; Chebyshev reference only). Energy: mean of the three per-seed arithmetic mean relative energy errors. $G$: paired geometric-mean ratio of the classical reference's error to the model's over the pooled 300 samples [95\% CI]; $W$: fraction of samples on which the model is more accurate than the reference. Rows for the other learned baselines of the original study are omitted here; they are kept in the archive and are not part of the comparisons of this paper.}\label{tab:core}
\begin{adjustbox}{max width=\linewidth}
\begin{tabular}{lllllll}\toprule
 & \multicolumn{3}{c}{channels (contrast $10^4$)} & \multicolumn{3}{c}{barriers (contrast $10^{-4}$)} \\ \cmidrule(lr){2-4}\cmidrule(lr){5-7}
model & energy & $G$ [95\% CI] & $W$ & energy & $G$ [95\% CI] & $W$ \\\midrule
OA-VCB (ours) & 2.75e-3 & 9.89 [9.57, 10.20] & 1.00 & 5.80e-2 & 2.00 [1.86, 2.14] & 0.88 \\
OA-VC (no deflation, plain output) & 3.46e-3 & 8.80 [8.45, 9.12] & 1.00 & 3.73e-1 & 0.61 [0.54, 0.70] & 0.43 \\
\midrule
classical reference (Chebyshev, $\nu=4$) & 2.57e-2 & 1 & -- & 8.11e-2 & 1 & -- \\
two-grid PCG, 6 iterations & 2.08e-3 & 14.13 [13.49, 14.81] & 1.00 & 5.05e-2 & 10.36 [8.76, 12.23] & 0.89 \\
$u_c$ alone & 3.73e-1 & 0.07 [0.07, 0.07] & 0.00 & 4.59e-1 & 0.07 [0.07, 0.08] & 0.00 \\
\bottomrule\end{tabular}
\end{adjustbox}
\end{table}

\begin{table}[tbp]\centering\footnotesize\setlength{\tabcolsep}{4pt}
\caption{Stronger classical smoothers and repeated classical passes, held-out seed 1, measured on one CPU (Intel(R) Xeon(R) Processor @ 2.80GHz, 2 threads; a different machine from Table~\ref{tab:cost}, so OA-VCB was re-timed): energy, the arithmetic mean relative energy error over the 100 test samples (the quantity plotted in Figure~\ref{fig:avt}); $G$, the paired geometric-mean ratio of each classical variant's error to OA-VCB's [95\% CI] (below 1: the classical variant is more accurate in this metric); $W$, the fraction of samples on which OA-VCB is more accurate; and the mean time per sample from $u_c$ with set-up, in ms (classical rows: per-sample minimum over the repetitions; OA-VCB: operator data and batched inference, minimum over the repetitions of the whole-test-set time divided by the number of samples). The eight- and sixteen-step rows exceed the residual-polynomial degree covered by Proposition~\ref{prop:containment} for the four-stage architecture; no learned model with that larger budget was trained.}\label{tab:avt}
\begin{adjustbox}{max width=\linewidth}
\begin{tabular}{lllll}\toprule
method & energy & $G$ [95\% CI] & $W$ & ms \\\midrule
\multicolumn{5}{l}{\emph{channels}: OA-VCB 2.55e-3, 268\,ms} \\
\quad reference, Chebyshev $\nu=4$ (Table~\ref{tab:core}) & 2.59e-2 & 10.23 [9.75, 10.71] & 1.00 & 83 \\
\quad reference, Chebyshev $\nu=8$ & 1.08e-2 & 4.23 [4.00, 4.46] & 1.00 & 79 \\
\quad reference, Chebyshev $\nu=16$ & 4.72e-3 & 1.83 [1.73, 1.93] & 0.99 & 85 \\
\quad reference, symmetric Gauss--Seidel, 4 sweeps & 2.06e-2 & 8.10 [7.69, 8.49] & 1.00 & 77 \\
\quad reference ($\nu=4$) applied twice & 3.50e-3 & 1.30 [1.21, 1.40] & 0.79 & 87 \\
\quad reference ($\nu=4$) applied 3 times & 5.51e-4 & 0.19 [0.17, 0.21] & 0.00 & 88 \\
\multicolumn{5}{l}{\emph{barriers}: OA-VCB 4.36e-2, 273\,ms} \\
\quad reference, Chebyshev $\nu=4$ (Table~\ref{tab:core}) & 7.51e-2 & 2.13 [1.90, 2.38] & 0.92 & 80 \\
\quad reference, Chebyshev $\nu=8$ & 6.47e-2 & 1.25 [1.09, 1.42] & 0.64 & 83 \\
\quad reference, Chebyshev $\nu=16$ & 5.87e-2 & 0.82 [0.69, 0.97] & 0.45 & 85 \\
\quad reference, symmetric Gauss--Seidel, 4 sweeps & 7.30e-2 & 2.11 [1.87, 2.36] & 0.92 & 80 \\
\quad reference ($\nu=4$) applied twice & 2.66e-2 & 0.29 [0.24, 0.35] & 0.13 & 84 \\
\quad reference ($\nu=4$) applied 3 times & 1.21e-2 & 0.05 [0.03, 0.06] & 0.02 & 88 \\
\bottomrule\end{tabular}
\end{adjustbox}
\end{table}

\begin{table}[tbp]\centering\footnotesize\setlength{\tabcolsep}{4pt}
\caption{Cost per test sample on one CPU (2 threads), channels, held-out seed 1, in ms. Every row starts from the assembled fine matrix; its assembly (1.4 / 5.4 ms, common to all methods) is not included. Per-operator set-up is included in every row; the sparsity patterns of the hierarchy, computed once per mesh (71 / 871 ms), and training are not. Classical rows: mean over the 100 samples of the per-sample minimum over 3 (L3) / 2 (L4) repetitions; OA-VCB operator data and inference: minimum over the repetitions of the whole-test-set time divided by 100. Iteration counts were determined beforehand, untimed, with the reference solution (a matched-accuracy benchmark, not a stopping criterion).}\label{tab:cost}
\begin{adjustbox}{max width=\linewidth}
\begin{tabular}{lll}\toprule
 & L3, $k_0=2$ (3{,}486 unknowns) & L4, $k_0=3$ (14{,}398) \\\midrule
\multicolumn{3}{l}{\emph{OA-VCB (one pass)}} \\
$u_c$ (level-$k_0$ Galerkin factor and solve) & 2.1 & 9.6 \\
operator data (hierarchy, Galerkin levels, $Z$) & 66.6 & 841.3 \\
inference, batch 8 (batch 1) & 96.0 (176.7) & 1{,}561.0 (1{,}461.8) \\
\textbf{total} (batch 8) & 164.8 & 2{,}411.8 \\
\midrule\multicolumn{3}{l}{\emph{single classical pass, error not matched (10 times larger)}} \\
classical reference, from $u_c$ & 60.1 & 769.1 \\
\midrule\multicolumn{3}{l}{\emph{iterative, run to OA-VCB's error on each sample [median; range of iterations]}} \\
two-grid PCG & 5.7 [6; 5--7] & 24.7 [6; 4--8] \\
SA-AMG PCG & 13.5 [5; 4--6] & 41.5 [6; 4--7] \\
PCG with the classical V(4,4) & 61.2 [3; 2--3] & 794.1 [3; 2--3] \\
\midrule\multicolumn{3}{l}{\emph{direct, to solver accuracy}} \\
sparse direct solve (SuperLU) & 8.2 & 48.1 \\
\bottomrule\end{tabular}
\end{adjustbox}
\end{table}

Table~\ref{tab:avt} gives the paired comparison of OA-VCB with stronger and repeated Chebyshev passes on the second platform
(held-out data seed $1$; Figure~\ref{fig:avt} shows the corresponding error--time curves), and Table~\ref{tab:cost} the cost
comparison with whole-system solvers on the first platform at L3 and L4. Neither was repeated for the Stage 1 operators; the
Stage 1 cost measurement is in Table~\ref{tab:s1cost}.

\section{Original-study design ablations}\label{app:ablations}
\begin{table}[tbp]\centering\footnotesize\setlength{\tabcolsep}{4pt}
\caption{Separately trained design variants, held-out data seed 1. Deflation, output used at inference (plain cycle or Galerkin projection), training loss and smoother class; mean energy error with each model's own output and reference/model $G$, $W$. OA-VCQ was not run on channels, where no near-kernel vectors are detected.}\label{tab:factorial}
\begin{adjustbox}{max width=\linewidth}
\begin{tabular}{llllllllll}\toprule
 & deflation & output & loss & smoothers & params & \multicolumn{2}{c}{barriers: energy, ($G$, $W$)} & \multicolumn{2}{c}{channels: energy, ($G$, $W$)} \\\midrule
OA-VC & no & plain & plain & learned & 20{,}672 & 2.33e-1 & 0.71, 0.44 & 4.08e-3 & 8.83, 0.99 \\
OA-VCZ & yes & plain & plain & learned & 20{,}672 & 7.51e-2 & 1.52, 0.75 & 3.52e-3 & 8.72, 0.99 \\
OA-VCQ & no & Galerkin & plain + proj. & learned & 20{,}672 & 1.07e-1 & 1.15, 0.65 & -- & -- \\
OA-VCB & yes & Galerkin & plain + proj. & learned & 20{,}672 & 4.36e-2 & 2.13, 0.92 & 2.55e-3 & 10.23, 1.00 \\
OA-VCL & yes & plain & plain & residual-linear & 21{,}600 & 1.51e-1 & 0.70, 0.16 & 1.11e-2 & 2.59, 0.97 \\
OA-VCBL & yes & Galerkin & plain + proj. & residual-linear & 21{,}600 & 6.76e-2 & 1.13, 0.69 & 1.01e-2 & 2.72, 0.99 \\
\bottomrule\end{tabular}
\end{adjustbox}
\end{table}

\begin{table}[tbp]\centering\footnotesize\setlength{\tabcolsep}{4pt}
\caption{Operator-adapted ($m=4$) vs geometric ($m=0$) transfers, separately trained OA-VCB models on the same data. OA-VCB error with each; paired $G$ (geometric-transfer error / adapted-transfer error) and $W$; the same ratio for the two classical references; and reference/learned $G$ for each choice of transfers.}\label{tab:transfers}
\begin{adjustbox}{max width=\linewidth}
\begin{tabular}{llllllll}\toprule
case & adapted & geometric & $G$ geom./adapt. & $W$ & references: geom./adapt. & reference/learned (adapt.) & reference/learned (geom.) \\\midrule
channels, seed 1 & 2.55e-3 & 5.02e-3 & 1.91 [1.84, 2.00] & 1.00 & 1.24 & 10.23 & 6.65 \\
channels, seed 2 & 2.79e-3 & 5.46e-3 & 1.92 [1.84, 2.00] & 1.00 & 1.24 & 9.76 & 6.30 \\
channels, seed 3 & 2.92e-3 & 5.66e-3 & 1.87 [1.80, 1.95] & 1.00 & 1.25 & 9.68 & 6.50 \\
barriers, seed 1 & 4.36e-2 & 1.13e-1 & 4.07 [3.50, 4.71] & 0.97 & 2.82 & 2.13 & 1.48 \\
\bottomrule\end{tabular}
\end{adjustbox}
\end{table}

Table~\ref{tab:factorial} compares separately trained design variants on held-out seed $1$ (plain or Galerkin output, with or
without deflation, residual-linear or nonlinear smoothers); Table~\ref{tab:transfers} compares operator-adapted ($m=4$) and
geometric ($m=0$) transfers. The Galerkin output with its additional loss term adds $G = 1.63$ without deflation and $1.40$ with it
on barriers; on channels OA-VCB is $1.17$ times more accurate than OA-VCZ. Without deflation, training for the projected output
improves the projected output ($0.121\to0.107$) but degrades the plain one ($0.233\to0.311$); with deflation both improve
($0.049\to0.044$ and $0.075\to0.066$). The gains are therefore not independent and additive. One OA-VCBL run (initialisation seed
$2$) ended with a $2.8$ times higher final training loss and a plain-cycle error of $5.9\times10^{-2}$ instead of
$1.07\times10^{-2}$, while its projected mean error is similar ($9.8\times10^{-3}$ against $1.01\times10^{-2}$). The adapted
transfers help the learned cycle more than the classical one ($1.24$ and $2.82$ for the references) at higher cost: operator data
$121$--$152$ instead of $11$--$14$\,ms per sample on the training platform's CPU. Because the models are trained separately, both
comparisons include the effect on optimisation.

\section{Repeated training and the output/loss ablation}\label{app:repeat}
\begin{table}[tbp]\centering\footnotesize\setlength{\tabcolsep}{4pt}
\caption{Repeated training with initialisation seeds 0 (the runs of the other tables), 1 and 2 on the same data. (a) OA-VCB: reference/OA-VCB $G$ ($W$) per data seed and initialisation seed, and the range of the mean energy error over the three trained models. (b) The residual-linear control OA-VCBL on channels, data seed 1: mean energy error of its projected (deployed) and plain cycle output, reference/OA-VCBL $G$ ($W$), and OA-VCBL/OA-VCB $G$ ($W$) with OA-VCB of the same initialisation seed (the two factors of Section~\ref{sec:ablation}). The ranges describe variability across training runs; the three models are evaluated on the same 100 test operators, so they are not independent test observations.}\label{tab:initseeds}
{\small (a) OA-VCB}\par\smallskip
\begin{adjustbox}{max width=\linewidth}
\begin{tabular}{llllll}\toprule
class & data seed & init.\ 0 & init.\ 1 & init.\ 2 & mean energy (range) \\\midrule
channels & 1 & 10.23 (1.00) & 10.38 (1.00) & 10.74 (1.00) & 2.43e-3--2.55e-3 \\
channels & 2 & 9.76 (1.00) & 9.94 (1.00) & 10.08 (1.00) & 2.72e-3--2.79e-3 \\
channels & 3 & 9.68 (1.00) & 9.96 (1.00) & 10.30 (1.00) & 2.79e-3--2.92e-3 \\
\midrule
barriers & 1 & 2.13 (0.92) & 2.09 (0.94) & 2.04 (0.91) & 4.36e-2--4.61e-2 \\
barriers & 2 & 1.93 (0.83) & 1.93 (0.85) & 1.96 (0.85) & 6.64e-2--6.96e-2 \\
barriers & 3 & 1.93 (0.89) & 1.92 (0.90) & 1.80 (0.86) & 6.06e-2--6.30e-2 \\
\bottomrule\end{tabular}
\end{adjustbox}
\par\medskip
{\small (b) OA-VCBL, channels, data seed 1}\par\smallskip
\begin{adjustbox}{max width=\linewidth}
\begin{tabular}{lllll}\toprule
init.\ seed & projected & plain & ref./OA-VCBL & OA-VCBL/OA-VCB \\\midrule
0 & 1.01e-2 & 1.07e-2 & 2.72 (0.99) & 3.76 (1.00) \\
1 & 1.01e-2 & 1.07e-2 & 2.72 (0.99) & 3.81 (1.00) \\
2 & 9.83e-3 & 5.87e-2 & 2.63 (1.00) & 4.08 (1.00) \\
\bottomrule\end{tabular}
\end{adjustbox}
\end{table}

\begin{table}[tbp]\centering\footnotesize\setlength{\tabcolsep}{4pt}
\caption{Plain and projected outputs of the same trained weights (mean energy error, held-out seed 1). Models with loss ``plain'' optimise the plain output only; models with loss ``plain + proj.'' optimise both output losses, and their projected output is the one deployed. In each row the other column is the same network evaluated with the other output.}\label{tab:factorial8}
\begin{adjustbox}{max width=\linewidth}
\begin{tabular}{lllllll}\toprule
class & model & deflation & loss & smoothers & plain output & projected output \\\midrule
barriers & OA-VC & no & plain & learned & 0.2329 & 0.1206 \\
channels & OA-VC & no & plain & learned & 0.0041 & 0.0035 \\
barriers & OA-VCZ & yes & plain & learned & 0.0751 & 0.0487 \\
channels & OA-VCZ & yes & plain & learned & 0.0035 & 0.0033 \\
barriers & OA-VCQ & no & plain + proj. & learned & 0.3111 & 0.1072 \\
barriers & OA-VCB & yes & plain + proj. & learned & 0.0663 & 0.0436 \\
channels & OA-VCB & yes & plain + proj. & learned & 0.0026 & 0.0026 \\
barriers & OA-VCL & yes & plain & residual-linear & 0.1508 & 0.0683 \\
channels & OA-VCL & yes & plain & residual-linear & 0.0111 & 0.0104 \\
barriers & OA-VCBL & yes & plain + proj. & residual-linear & 0.1518 & 0.0676 \\
channels & OA-VCBL & yes & plain + proj. & residual-linear & 0.0107 & 0.0101 \\
\bottomrule\end{tabular}
\end{adjustbox}
\end{table}

Table~\ref{tab:initseeds} reports the three initialisation seeds of OA-VCB on every held-out data set and of the residual-linear
control OA-VCBL on channels (data seed 1). One OA-VCBL run is not a repeat of the others in the optimisation sense: initialisation
seed 1 reaches almost the same final training loss as seed 0, while seed 2 ends with a $2.8$ times higher final training loss and a
much larger plain-cycle error, with a similar projected mean error and a slightly lower paired gain (Section~\ref{sec:ablation}).
Matching initialisation seed numbers across architectures index runs; they do not match optimisation trajectories.
Table~\ref{tab:factorial8} evaluates each trained network of Table~\ref{tab:factorial} with both outputs, which separates the
output used at inference from the training objective. Projecting the output of a network trained without the projection term
also helps on its own (OA-VC $0.233\to0.121$, OA-VCZ $0.075\to0.049$ on barriers). One channel sample (id 320) is difficult for
every model without the trained Galerkin output and for OA-VCBL; we do not interpret it further.

\section{Sensitivity of the detector thresholds and of the transfer steps}\label{app:sensitivity}
Table~\ref{tab:sensitivity} varies two choices that were fixed during development, using the classical reference, which needs
no retraining (\path{diagnostics_cost/sensitivity_detector_transfers.py}). These are subsequent checks; the main experiments keep
the development values $\theta = 0.1$, $\tau = 0.05$. Recall counts islands (202 in the 300 held-out operators; an operator may
contain several) whose indicator lies in $\operatorname{span}Z$ to a relative Euclidean distance below $10^{-8}$. Precision counts
detected vectors with at least $90\%$ of their support inside one island, out of all detected vectors; there is no one-to-one
matching, so the two numerators count different objects. (a) The threshold $\theta$ matters more than $\tau$. At the default
$\tau = 0.05$, $\theta = 0.05$, $0.1$ and $0.2$ recover $202$, $195$ and $81$ of the $202$ islands; $\theta = 0.05$ does so with one
detected vector outside an island, and $\theta = 0.2$ more than doubles the reference's error on island samples. Varying $\tau$
in $[0.02, 0.1]$ changes the reference's mean error less than varying $\theta$, but it can change the detected space: at
$\theta = 0.1$ the recall rises from $195$ to $201$ and the number of detected vectors from $204$ to $224$, and at $\theta = 0.2$
the number of detected vectors grows from $116$ to $498$ while the error changes little. Only at $\theta = 0.05$ are the tested
values of $\tau$ without effect. (b) Re-evaluating the trained barrier models with the $Z$ of $\theta = 0.05$ (only deflation
and the Galerkin output change) lowers the mean error of OA-VCB by $4.6\%$ and that of the reference by $3.1\%$; the pooled
paired gain changes from $2.00$ to $2.03$ (\path{diagnostics_cost/detector_theta_learned.py}). (c) More Jacobi steps in the
transfers improve the classical reference, by $7\%$ on channels and $30\%$ on barriers from $m = 4$ to $m = 8$, while the
operator complexity grows from $4.5$ to $7.4$ and the set-up time fourfold; with $m = 1$ or $2$ the reference loses $7$--$11\%$ on
channels and $38$--$77\%$ on barriers. Retraining OA-VCB with $m = 8$ (held-out seed 1, Table~\ref{tab:m8}) keeps the learned
margin over the correspondingly stronger reference: $G = 11.17$ on channels ($10.23$ with $m = 4$) and $2.01$ on barriers
($2.13$), with $W = 1.00$ and $0.90$. Measured against the $m = 4$ reference, the $m = 8$ model gains $12.0$ and $2.88$.
\begin{table}[tbp]\centering\footnotesize\setlength{\tabcolsep}{4pt}
\caption{Transfers with $m=8$ Jacobi steps (held-out data seed 1, initialisation seed 0): OA-VCB trained with $m=8$ and its classical reference with $m=8$; mean energy errors, paired $G$ [95\% CI] and $W$; for comparison the $m=4$ margin and the improvement of the classical reference from $m=4$ to $m=8$.}\label{tab:m8}
\begin{adjustbox}{max width=\linewidth}
\begin{tabular}{lllllll}\toprule
class & OA-VCB ($m=8$) & reference ($m=8$) & $G$ [95\% CI] & $W$ & $G$ at $m=4$ & ref.\ $m=4$/$m=8$ \\\midrule
channels & 2.14e-3 & 2.42e-2 & 11.17 [10.74, 11.63] & 1.00 & 10.23 [9.75, 10.71] & 1.07 \\
barriers & 3.48e-2 & 5.19e-2 & 2.01 [1.79, 2.23] & 0.90 & 2.13 [1.90, 2.38] & 1.44 \\
\bottomrule\end{tabular}
\end{adjustbox}
\end{table}

\begin{table}[tbp]\centering\footnotesize\setlength{\tabcolsep}{4pt}
\caption{Sensitivity of two fixed choices (classical reference unless stated; L3, $k_0=2$). (a) Detector thresholds on the barrier test sets of held-out seeds 1--3 (300 operators, 202 coefficient islands in total; an operator may contain several). Recall: islands whose indicator lies in $\operatorname{span}Z$ (relative Euclidean distance below $10^{-8}$), out of all islands. Precision: detected vectors with at least $90\%$ of their support inside one island, out of all detected vectors (no one-to-one matching, so the two numerators count different objects and need not agree). Mean reference energy error over all samples and over the island samples. (b) The trained barrier OA-VCB models (trained with $\theta=0.1$) evaluated with $Z$ from $\theta=0.1$ and $0.05$ ($\tau=0.05$), pooled seeds 1--3: mean energy errors, reference/OA-VCB $G$ [95\% CI] and $W$. (c) Number $m$ of Jacobi steps in the transfers, held-out seed 1: operator complexity $\sum_k\operatorname{nnz}(A_k)/\operatorname{nnz}(A_L)$ (median), hierarchy build time (median, one CPU), reference energy error and the paired geometric-mean ratio of its errors to those at $m=4$; the learned cycle was trained only with $m=4$ and $m=0$ (Table~\ref{tab:transfers}).}\label{tab:sensitivity}
{\small (a) detector thresholds, classical reference}\par\smallskip
\begin{adjustbox}{max width=\linewidth}
\begin{tabular}{llllll}\toprule
$\theta$ & $\tau$ & recall & precision & energy, all & energy, island samples \\\midrule
0.05 & 0.02 & 202/202 & 202/203 & 7.86e-2 & 1.65e-1 \\
0.05 & 0.05 & 202/202 & 202/203 & 7.86e-2 & 1.65e-1 \\
0.05 & 0.1 & 202/202 & 202/203 & 7.86e-2 & 1.65e-1 \\
0.1 & 0.02 & 195/202 & 197/204 & 8.11e-2 & 1.72e-1 \\
0.1 & 0.05 (used) & 195/202 & 199/211 & 8.11e-2 & 1.72e-1 \\
0.1 & 0.1 & 201/202 & 208/224 & 7.86e-2 & 1.66e-1 \\
0.2 & 0.02 & 79/202 & 95/116 & 1.97e-1 & 4.43e-1 \\
0.2 & 0.05 & 81/202 & 131/211 & 1.95e-1 & 4.39e-1 \\
0.2 & 0.1 & 86/202 & 187/498 & 1.91e-1 & 4.30e-1 \\
\bottomrule\end{tabular}
\end{adjustbox}
\par\medskip
{\small (b) trained OA-VCB with the detector of a smaller $\theta$}\par\smallskip
\begin{adjustbox}{max width=\linewidth}
\begin{tabular}{lllll}\toprule
$\theta$ & OA-VCB & reference & $G$ [95\% CI] & $W$ \\\midrule
0.1 & 5.80e-2 & 8.11e-2 & 2.00 [1.86, 2.14] & 0.88 \\
0.05 & 5.53e-2 & 7.86e-2 & 2.03 [1.90, 2.17] & 0.89 \\
\bottomrule\end{tabular}
\end{adjustbox}
\par\medskip
{\small (c) Jacobi steps in the transfers, classical reference}\par\smallskip
\begin{adjustbox}{max width=\linewidth}
\begin{tabular}{lllllll}\toprule
$m$ & op.\ complexity & build (ms) & channels & $G$ vs $m=4$ & barriers & $G$ vs $m=4$ \\\midrule
0 & 1.27 & 2 & 3.23e-2 & 1.24 & 1.23e-1 & 2.82 \\
1 & 1.82 & 6 & 2.88e-2 & 1.11 & 1.02e-1 & 1.77 \\
2 & 2.63 & 13 & 2.78e-2 & 1.07 & 9.23e-2 & 1.38 \\
4 (used) & 4.54 & 56 & 2.59e-2 & 1.00 & 7.51e-2 & 1.00 \\
8 & 7.44 & 232 & 2.42e-2 & 0.93 & 5.19e-2 & 0.70 \\
\bottomrule\end{tabular}
\end{adjustbox}
\end{table}

\section{Harmonic enrichment: protocol, amendment and results}\label{app:harmonic}
\begin{table}[tbp]\centering\footnotesize\setlength{\tabcolsep}{4pt}
\caption{Post-hoc harmonic enrichment of the frozen OA-VCB predictions (barriers; patch radius 2, fixed on the development set; the fresh confirmation set has 100 new operators per checkpoint). (a) Detector-positive samples / all samples; mean energy error before and after the local correction; $R = 1-\overline{e}_{\mathrm{after}}/\overline{e}_{\mathrm{before}}$ in \% [95\% CI]; $G_{24}$: paired ratio of the patch-indicator control's error to the harmonic correction's [95\% CI]. (b) Per-sample captured fraction of squared energy error, median [quartiles]; on the fixed global-reference subset (number of samples), $R$ of the local and of the global correction, in \%.}\label{tab:harmonic}
{\small (a) Primary outcome}\par\smallskip
\begin{adjustbox}{max width=\linewidth}
\begin{tabular}{llllll}\toprule
set & samples & $e$ before & $e$ after & $R$ (\%) & $G_{24}$ \\\midrule
development (seed 1) & 42/100 & 0.0956 & 0.0876 & 8.3 [5.1, 12.3] & 1.10 [1.06, 1.13] \\
retrospective (seeds 2, 3) & 92/200 & 0.1352 & 0.1170 & 13.5 [9.6, 17.9] & 1.15 [1.11, 1.19] \\
fresh confirmation & 89/200 & 0.1219 & 0.1043 & 14.4 [9.8, 19.5] & 1.13 [1.09, 1.16] \\
\bottomrule\end{tabular}
\end{adjustbox}
\par\medskip
{\small (b) Captured fraction and local versus global extension}\par\smallskip
\begin{adjustbox}{max width=\linewidth}
\begin{tabular}{lllll}\toprule
set & captured fraction & subset size & local $R$ (\%) & global $R$ (\%) \\\midrule
development (seed 1) & 0.09 [0.03, 0.31] & 20 & 9.0 & 17.6 \\
retrospective (seeds 2, 3) & 0.13 [0.03, 0.40] & 40 & 12.2 & 16.2 \\
fresh confirmation & 0.12 [0.02, 0.33] & 40 & 10.9 & 20.9 \\
\bottomrule\end{tabular}
\end{adjustbox}
\end{table}

\paragraph{Correction.} For a frozen prediction $x$ with residual $r = F - Ax$ and a small matrix $Z$ built from $A$ alone
(independent columns), the correction
\begin{gather*}
  x_+ = x + Z(Z^\top AZ)^{-1}Z^\top r,\\
  \|u-x_+\|_A^2 = \|u-x\|_A^2 - (Z^\top r)^\top(Z^\top AZ)^{-1}(Z^\top r),
\end{gather*}
is the $A$-orthogonal projection of the error onto $\operatorname{span}Z$ (exact arithmetic): a sequential correction of the
frozen error, not a joint re-optimisation over the cycle's own directions together with $Z$; the networks are not retrained. An
earlier classical pilot had suggested that an $A$-harmonic tail approximates the variation of the slowest mode across the
barrier; a low eigenmode need not satisfy $Az=0$ outside the island, but the enrichment constructed here is harmonic by
definition. $Z$ contains the indicators and, for each island, its $A$-harmonic extension into a patch of $\rho$ graph rings (zero
outside). A control of equal support replaces the harmonic extension by the indicator of the patch, and a reference uses the
global harmonic extension. Island samples here are those on which the detector of Section~\ref{sec:deflation} returns at least
one vector (42, 92 and 89 samples in the three sets, against 39, 88 and 84 with the geometric label). We report
$R = 1-\overline{e}_{\mathrm{after}}/\overline{e}_{\mathrm{before}}$, the reduction of the mean relative energy error; $G_{24}$,
the paired ratio of the control's error to the harmonic correction's; and the per-sample captured fraction of squared energy
error, $\kappa = 1 - \|u-x_+\|_A^2/\|u-x\|_A^2$.

\paragraph{Protocol.} The protocol (version 2) fixed, before any enrichment number: the frozen checkpoints and the inference
path; the three data sets and their roles (development, retrospective, fresh confirmation); the island label (the detector of
Section~\ref{sec:deflation} with $\theta=0.1$, $\tau=0.05$; an independent geometric label reported only); the corrections
(indicators; local harmonic extensions for $\rho\in\{1,2,4,8\}$; the patch-indicator control; global extensions on the first 20
detector-positive samples per checkpoint, i.e.\ per data seed: 20 / 40 / 40 samples in the three sets); the numerics (columns
normalised in the energy norm, greedy pivoted Cholesky of the scaled Gram matrix with pivot threshold $10^{-10}$; gain computed
as $2d^\top r - d^\top Ad$ and checked against the change of squared error, agreement to $10^{-10}$ relative); the radius rule
($\rho = 2$, the smallest radius within $10\%$ of the best on the seed-1 test set); and the outcome thresholds, ``strong support''
($R\ge0.5$ and $G_{24}\ge1.25$), ``partial support'' ($R\ge0.2$) and ``little support'' otherwise, decided on a fresh
confirmation set of $100$ new coefficient fields per checkpoint, never generated before. Seeds 2 and 3 served as a retrospective
evaluation: they were held out from training and from the radius choice, but their model errors had been inspected before.
\emph{Amendment (version 3), made after the reproduction check and before any enrichment number.} The planned check required the
CPU re-evaluation of the stored predictions to agree with the GPU records to $10^{-4}$ per sample. Means agreed to six digits on
all three seeds and samples without islands to $3\times10^{-5}$, but island samples differed by up to $1.4\times10^{-3}$. On the
same CPU, changing only the batch size moved island samples by up to $1.1\times10^{-4}$ and evaluating the network in double
precision by up to $1.75\times10^{-4}$, consistent with single-precision round-off amplified by near-singular operators, although
these checks do not show that round-off accounts for the whole CPU/GPU difference. The study therefore used the CPU
single-precision path; the criterion was replaced by agreement of the means to $10^{-5}$ and of the non-island samples to $10^{-4}$.

\paragraph{Further results.} The captured fraction is heterogeneous (median $0.12$, quartiles $0.02$ and $0.33$ on the fresh
operators). The classical reference behaves similarly ($R = 18.5\%$ locally on the fresh operators; $15.6\%$ local and $29.2\%$
global on the reference subset; per-seed values, including one seed with partial support, in Table~\ref{tab:harmonicseeds}). The
indicator columns alone gave numerically negligible gains (at most $10^{-19}$ relative), consistent with both outputs already
containing their span. As an observation on the 42 detector-positive development samples, the means of the per-sample regional
fractions of squared energy error are $28\%$ inside the islands, $31\%$ in the barriers and $41\%$ in the connected conductive
region; after the correction the barrier fraction drops to $20\%$. Building the local basis took about $0.6$\,ms and applying the
projection, including the study's diagnostics, about $0.5$\,ms per island sample (median; the detection, $1.5$\,ms, is already
part of OA-VCB).
\begin{table}[tbp]\centering\footnotesize\setlength{\tabcolsep}{4pt}
\caption{Harmonic enrichment per seed, OA-VCB and classical reference: detector-positive samples, $R$ in \% [95\% CI], $G_{24}$, and the outcome category of the protocol.}\label{tab:harmonicseeds}
\begin{adjustbox}{max width=\linewidth}
\begin{tabular}{lllllll}\toprule
set & seed & predictor & samples & $R$ (\%) & $G_{24}$ & outcome \\\midrule
development & 1 & OA-VCB & 42 & 8.3 [5.1, 12.3] & 1.10 & little \\
development & 1 & reference & 42 & 18.9 [14.8, 23.7] & 1.24 & little \\
retrospective & 2 & OA-VCB & 49 & 17.9 [11.6, 24.5] & 1.19 & little \\
retrospective & 2 & reference & 49 & 16.4 [12.2, 21.0] & 1.22 & little \\
retrospective & 3 & OA-VCB & 43 & 8.4 [5.3, 11.8] & 1.11 & little \\
retrospective & 3 & reference & 43 & 14.3 [9.9, 19.4] & 1.21 & little \\
fresh confirmation & 2 & OA-VCB & 37 & 17.8 [9.8, 26.1] & 1.15 & little \\
fresh confirmation & 2 & reference & 37 & 22.7 [18.0, 27.8] & 1.31 & partial \\
fresh confirmation & 3 & OA-VCB & 52 & 11.5 [6.9, 16.6] & 1.11 & little \\
fresh confirmation & 3 & reference & 52 & 15.3 [10.6, 20.8] & 1.22 & little \\
\bottomrule\end{tabular}
\end{adjustbox}
\end{table}

\section{Load tests}\label{app:loads}
\paragraph{Loads.} The random-load distribution used for training (Section~\ref{sec:k0}) draws a constant $c\sim U(-1,1)$ plus a
sine expansion $\sum_{p,q=0}^{3}K_{pq}\sin((p+1)\pi x)\sin((q+1)\pi y)/(1+p+q)$ with $K_{pq}\sim N(0,1)$, to which, with
probability one half, a Gaussian of width $U(0.02,0.08)$ and amplitude $U(2,10)$ centred in $[0.1,0.9]^2$ is added. The smooth
test uses the same finite sine span with a different coefficient distribution and no constant term. The localised test is a
pure, unit-amplitude Gaussian source of fixed width $0.03$ centred in $[0.15,0.85]^2$; its width lies inside the training
interval, so it changes the composition of the forcing rather than its width, and by positive homogeneity its amplitude is
immaterial. The load values $f=\pm1$ lie inside the support of the training distribution, so they test coverage of that
distribution's extreme members; neither test is extrapolation beyond the support, and a different generating procedure does not
by itself place a load outside it. A miss of the per-seed criterion by one sample is a threshold outcome, not evidence of a sharp
boundary.
\begin{table}[tbp]\centering\footnotesize\setlength{\tabcolsep}{4pt}
\caption{The fixed-load ($f=1$) OA-VCB checkpoints of Table~\ref{tab:core} applied to other loads (channels, held-out seeds 1--3, 300 operators); $u_c$ is recomputed for each load. Mean energy errors, reference/OA-VCB $G$ per seed and pooled, $W$, and the number of samples on which OA-VCB is worse than $u_c$.}\label{tab:rhs}
\begin{adjustbox}{max width=\linewidth}
\begin{tabular}{llllllll}\toprule
load & $u_c$ & OA-VCB & reference & $G$ per seed & $G$ pooled [95\% CI] & $W$ & $>u_c$ \\\midrule
$f=1$ (training load) & 0.373 & 2.75e-3 & 2.57e-2 & 10.23 / 9.76 / 9.68 & 9.89 [9.57, 10.20] & 1.00 & 0 \\
$f=-1$ & 0.373 & 1.37e-1 & 2.57e-2 & 0.30 / 0.28 / 0.11 & 0.21 [0.20, 0.23] & 0.00 & 0 \\
$f=3$ & 0.373 & 2.75e-3 & 2.57e-2 & 10.23 / 9.76 / 9.68 & 9.89 [9.57, 10.20] & 1.00 & 0 \\
random-load distribution (Sec.~\ref{sec:k0}) & 0.397 & 1.01e-1 & 2.59e-2 & 0.66 / 0.74 / 0.45 & 0.60 [0.50, 0.72] & 0.37 & 0 \\
smooth, both signs & 0.386 & 1.55e-1 & 2.49e-2 & 0.25 / 0.29 / 0.12 & 0.20 [0.19, 0.23] & 0.06 & 0 \\
localised source & 0.635 & 1.08e-1 & 2.21e-2 & 0.38 / 0.54 / 0.42 & 0.44 [0.38, 0.51] & 0.32 & 0 \\
\bottomrule\end{tabular}
\end{adjustbox}
\end{table}

\begin{table}[tbp]\centering\footnotesize\setlength{\tabcolsep}{4pt}
\caption{The OA-VCB checkpoints trained on the random-load distribution (Section~\ref{sec:k0}; one model per held-out data seed 1--3) on the loads of Table~\ref{tab:rhs}, same operators. (a) Pooled over the 300 operators: mean energy errors of the random-load models and of the reference, reference/random-model $G$ [95\% CI] and $W$, and samples worse than $u_c$. (b) Per seed (1 / 2 / 3): reference/random-model $G$ and $W$; mean error of the fixed-load models of Table~\ref{tab:rhs}; and pooled fixed-load/random-load $G$ and $W$ (above 1 and above 0.5 favour the random-load models). The prespecified criterion per seed and load family was $G\ge2$ and $W\ge0.8$.}\label{tab:rhsrandom}
{\small (a) Pooled}\par\smallskip
\begin{adjustbox}{max width=\linewidth}
\begin{tabular}{llllll}\toprule
load & random-load & reference & $G$ [95\% CI] & $W$ & $>u_c$ \\\midrule
$f=1$ & 3.99e-3 & 2.57e-2 & 6.87 [6.62, 7.13] & 1.00 & 0 \\
$f=-1$ & 4.21e-3 & 2.57e-2 & 6.51 [6.27, 6.75] & 1.00 & 0 \\
$f=3$ (homogeneity check) & 3.99e-3 & 2.57e-2 & 6.87 [6.62, 7.13] & 1.00 & 0 \\
random-load distribution (Sec.~\ref{sec:k0}) & 4.43e-3 & 2.59e-2 & 6.22 [5.99, 6.44] & 1.00 & 0 \\
smooth, both signs & 4.00e-3 & 2.49e-2 & 6.38 [6.20, 6.56] & 1.00 & 0 \\
localised source & 1.25e-2 & 2.21e-2 & 2.51 [2.25, 2.79] & 0.84 & 0 \\
\bottomrule\end{tabular}
\end{adjustbox}
\par\medskip
{\small (b) Per seed, and against the fixed-load models}\par\smallskip
\begin{adjustbox}{max width=\linewidth}
\begin{tabular}{llllll}\toprule
load & $G$ per seed & $W$ per seed & fixed-load & fixed/random $G$ & $W$ \\\midrule
$f=1$ & 6.90 / 6.75 / 6.98 & 1.00 / 1.00 / 1.00 & 2.75e-3 & 0.70 & 0.01 \\
$f=-1$ & 6.67 / 6.39 / 6.47 & 1.00 / 1.00 / 1.00 & 1.37e-1 & 30.8 & 1.00 \\
$f=3$ (homogeneity check) & 6.90 / 6.75 / 6.98 & 1.00 / 1.00 / 1.00 & 2.75e-3 & 0.70 & 0.01 \\
random-load distribution (Sec.~\ref{sec:k0}) & 6.27 / 6.05 / 6.34 & 1.00 / 1.00 / 1.00 & 1.01e-1 & 10.3 & 0.86 \\
smooth, both signs & 6.41 / 6.31 / 6.42 & 1.00 / 1.00 / 1.00 & 1.55e-1 & 31.1 & 1.00 \\
localised source & 2.80 / 2.15 / 2.61 & 0.89 / 0.79 / 0.83 & 1.08e-1 & 5.7 & 0.99 \\
\bottomrule\end{tabular}
\end{adjustbox}
\end{table}

\section{Further settings}\label{app:settings}
\begin{table}[tbp]\centering\footnotesize\setlength{\tabcolsep}{4pt}
\caption{Further settings: mean energy error of OA-VCB and of its classical reference, paired reference/OA-VCB $G$ [95\% CI] and $W$, and the per-sample number of two-grid PCG iterations, from the same starting point, at which PCG first reaches OA-VCB's energy error (median, range; computed iteration histories: 12 for $k_0=0$, 16 otherwise). The anisotropic tensor-coefficient family is reported in Table~\ref{tab:aniso}.}\label{tab:settings}
\begin{adjustbox}{max width=\linewidth}
\begin{tabular}{llllll}\toprule
setting & OA-VCB & reference & $G$ [95\% CI] & $W$ & PCG its \\\midrule
$k_0=0$ (no external coarse solution), seed 1 & 7.62e-3 & 8.52e-2 & 11.39 [10.95, 11.83] & 1.00 & 5 (4--7) \\
$k_0=0$ (no external coarse solution), seed 2 & 7.58e-3 & 8.04e-2 & 10.91 [10.35, 11.44] & 1.00 & 5 (4--6) \\
$k_0=0$ (no external coarse solution), seed 3 & 8.13e-3 & 8.55e-2 & 10.82 [10.40, 11.26] & 1.00 & 5 (5--7) \\
finer hierarchy (L4, $k_0=3$), seed 1 & 2.21e-3 & 2.01e-2 & 9.72 [9.04, 10.33] & 1.00 & 6 (4--8) \\
polygonal domain, seed 1 & 3.54e-3 & 3.30e-2 & 9.39 [8.89, 9.86] & 1.00 & 6 (5--8) \\
random loads (in distribution), seed 1 & 4.50e-3 & 2.65e-2 & 6.27 [5.89, 6.65] & 1.00 & 5 (4--7) \\
random loads (in distribution), seed 2 & 4.26e-3 & 2.51e-2 & 6.05 [5.72, 6.36] & 1.00 & 5 (4--7) \\
random loads (in distribution), seed 3 & 4.54e-3 & 2.61e-2 & 6.34 [5.90, 6.77] & 1.00 & 6 (4--7) \\
\bottomrule\end{tabular}
\end{adjustbox}
\end{table}

\begin{table}[tbp]\centering\footnotesize\setlength{\tabcolsep}{4pt}
\caption{Supplementary scope check: anisotropic tensor-coefficient family (meandering tows, $k_L=100$, $k_T=5$ against a unit background, insulated holes), one separately trained OA-VCB per data seed with the original protocol and initialisation seed $0$; $100$ test operators per seed. Mean relative energy errors, paired $G$ of the degree-four Chebyshev reference to OA-VCB with the 95\% bootstrap interval, and $W$. Data seed $0$ is the development seed, so the three rows are three configurations including the development one, not three independent held-out replications. Stronger classical smoothing budgets, Krylov-space controls and matched-accuracy solver costs were not evaluated for this family.}\label{tab:aniso}
\begin{tabular}{llllll}\toprule
data seed & status & OA-VCB & reference & $G$ [95\% CI] & $W$ \\\midrule
0 & development & 6.21e-03 & 2.10e-02 & 3.50 [3.33, 3.68] & 1.00 \\
1 & subsequent evaluation & 6.01e-03 & 2.04e-02 & 3.51 [3.33, 3.69] & 1.00 \\
2 & subsequent evaluation & 6.17e-03 & 2.12e-02 & 3.58 [3.40, 3.77] & 1.00 \\
\bottomrule\end{tabular}\end{table}

\paragraph{No external coarse solution ($k_0 = 0$).} The networks start from $u=0$; the network and the reference still contain
their exact 28-node coarsest solve. Two-grid PCG from zero first reaches OA-VCB's error after a median of $5$ iterations. 
\paragraph{Finer hierarchy.} $L=4$, $k_0=3$, 14{,}398 unknowns, one held-out seed; two-grid PCG first reaches OA-VCB's error after a
median of $6$ iterations (range $4$--$8$). This is not a mesh-refinement study at fixed information ($k_0$ changes with $L$).
\paragraph{Refinement without retraining.} The L4-trained model was applied unchanged to the same coefficient draws on the next
two refined meshes of held-out seed 1, with the reference on the same hierarchy. At L5 (58{,}494 unknowns, $k_0 = 4$; the
four-level cycle now ends on mesh level 1 instead of 0) it keeps $G = 8.12$ [7.50, 8.67] with $W = 1$ on all 100 test samples
(smallest per-sample ratio $1.65$) and is never worse than $u_c$; paired with the L4 results for the same coefficient draws
($G = 9.72$), its per-sample gain is $0.84$ [0.76, 0.91] times the L4 gain, smaller at L5 on $81\%$ of the samples. At L6
(235{,}774 unknowns, $k_0 = 5$, bottom on mesh level 2) the evaluation was stopped by the session time limit of the computing
platform after the first 21 of the 24 planned test samples (ids 300--320) and wrote no per-sample record; the values are taken from
the per-sample lines of the run log (four significant digits; rounding moves $G$ by less than $0.005$). On these completed cases
$G = 4.18$ [3.43, 4.88] with $W = 1$ and a smallest per-sample ratio of $1.06$, about half the L5 gain on the same draws (paired
ratio $0.51$ [0.46, 0.56]). The interval describes the completed cases; it does not account for the missing cases. The original
logs and the L5 record are kept with the code archive (\path{Results_kaggle_session 11/}). A model trained at L3 does not keep its margin
at L4 ($G = 1.95$, $W = 0.95$, 20 samples; $1.97$ with the cycle depth matched). In this evaluation, operator data and inference took
about $64$\,s per sample at L5 and $1{,}370$\,s at L6 on one CPU shared by concurrent jobs; these are not comparable with
Table~\ref{tab:cost}.
\paragraph{Polygonal domain.} A non-convex polygon with a re-entrant corner, a different unstructured coarse mesh with a 45-node
coarsest level: $G = 9.39$, $W=1$, one held-out seed.
\paragraph{Random loads, in distribution.}\label{sec:k0} Trained and tested on the distribution of \appref{app:loads}, one training run per data
seed: $G = 6.27$ [5.89, 6.65] on seed 1, which meets the outcome rule fixed before the run (``strong'' if $G\ge2$ and $W\ge0.8$), and
$6.05$ and $6.34$ on seeds 2 and 3, with the same rule fixed beforehand for both, all with $W = 1$.
\paragraph{Anisotropic tensor coefficient.} Meandering fibre tows with $K = k_T I + (k_L - k_T)\,t\,t^\top$ along the band direction
$t$, $k_L = 100$, $k_T = 5$, unit background, insulated holes: $G = 3.50$, $3.51$ and $3.58$ on data seeds $0$, $1$ and $2$, with $W=1$ (Table~\ref{tab:aniso}). Data seed $0$ was the development
seed of the whole study, so the three rows are three configurations including the development one, not three independent
held-out replications; each used the original protocol with initialisation seed $0$ and a separate training run. We do not
certify the training and model-selection history of this family beyond that statement. Here the reference is already more accurate than the plain
classical cycle (ratio of mean errors $1.37$). The detector returns at least one vector on $9$--$10\%$ of the samples (at most two);
with insulated holes the coarsest level has 52 nodes.
\paragraph{Plane-strain elasticity (preliminary).}\label{sec:elasticity}
\begin{table}[tbp]\centering\footnotesize\setlength{\tabcolsep}{4pt}
\caption{Plane-strain elasticity (stiff bands, $E_f/E_m=100$, traction-free pores; one data set, test samples 300--399; preliminary). Mean energy error, and paired reference/model $G$ [95\% CI] and $W$. OA-VCB has one training run. Rows for the learned baselines of the original study are omitted here and kept in the archive. Two-grid PCG first reaches OA-VCB's per-sample energy error after a median of 11 iterations (range 4--32).}\label{tab:elasticity}
\begin{adjustbox}{max width=\linewidth}
\begin{tabular}{lrlll}\toprule
model & params & energy & $G$ [95\% CI] & $W$ \\\midrule
OA-VCB (ours) & 20{,}672 & 2.47e-2 & 2.32 [2.20, 2.44] & 1.00 \\
\midrule
classical reference & -- & 5.50e-2 & 1 & -- \\
classical V(4,4), no proj. & -- & 7.53e-2 & 0.72 [0.71, 0.73] & 0.00 \\
$u_c$ alone & -- & 3.87e-1 & 0.13 [0.13, 0.14] & 0.00 \\
\bottomrule\end{tabular}
\end{adjustbox}
\end{table}

The same construction applies to vector problems on the degree-of-freedom graph: the prolongations are $P_k\otimes I_2$ with the
Jacobi steps of \eqref{eq:jacobi} on the vector operator, which reproduce the rigid-body modes where the hypotheses of
Proposition~\ref{prop:constants} hold (\appref{app:impl}). We consider plane strain (Poisson ratio $0.3$) on the perforated square,
clamped on the outer boundary, traction-free on the holes, with a downward body force and a Young's modulus that is $100$ times
larger in four meandering bands than in the matrix; $L=3$ (7{,}356 unknowns) and $k_0=2$. The scalar island detector does not
apply, and the bands reach the clamped boundary, so OA-VCB is used without deflation, with its Galerkin output and otherwise the
configuration of the scalar problems (no tuning for elasticity). Its classical reference is the same Chebyshev V$(4,4)$ cycle with
Galerkin output; the contained fixed-interval Chebyshev member is slightly more accurate than it and the contained Jacobi member
$1.82$ times less accurate (Table~\ref{tab:contained}). This experiment has one data set with a held-out test split, no independent
data-seed replications and one OA-VCB training run. OA-VCB is more accurate than its reference on every test sample,
$G = 2.32$ [2.20, 2.44], and two-grid PCG first reaches its energy error after a median of $11$ iterations (range $4$--$32$;
Table~\ref{tab:elasticity}). One preliminary case does not establish broad applicability to mechanics.

\section{Computing platforms and reproducibility}\label{app:platform}
\paragraph{Platforms.} Training, the stored evaluation records and the GPU inference times: NVIDIA T4 GPUs (Kaggle). Cost table
(Table~\ref{tab:cost}): one Intel Xeon processor at 2.1\,GHz, 2 threads. Stage 1 timings (Table~\ref{tab:s1cost}): one Intel Xeon
processor at 2.10\,GHz, 2 threads (\appref{app:stage1}). Stronger-smoother measurement (Table~\ref{tab:avt},
Figure~\ref{fig:avt}): a different machine, one Intel Xeon processor at 2.80\,GHz, 2 threads, on which OA-VCB was re-timed; its
CPU re-evaluation reproduces the stored OA-VCB errors to $5.6\times10^{-6}$ relative on channels and, on barriers, to within the
CPU/GPU round-off envelope documented in \appref{app:harmonic} (one sample at $1.4\times10^{-3}$, median $3\times10^{-6}$).
Timing scopes: classical rows report the mean over samples of the per-sample minimum over repetitions; OA-VCB's operator data
and inference report the minimum over repetitions of the whole-test-set time divided by the number of samples (inference in
batches of eight). Software: Python 3.13, PyTorch 2.14 (CPU), NumPy 2.5, SciPy 1.18 with OpenBLAS 0.3.34, PyAMG 5.3, Matplotlib
3.11; sparse direct solves with SuperLU.
\paragraph{Versions and regeneration.} Every training and evaluation record stores the hash \texttt{f8653612eb84}, an MD5
prefix of the top-level model and training sources computed by \path{run_identity.py}; it identifies the model code of every
stored record, not the diagnostic scripts in \path{diagnostics_cost/} and \path{mathcheck/}, the Kaggle drivers in
\path{kaggle_runs/}, the Stage 1 scripts in \path{ext/} or the table and figure generators. The archived release
\cite{ranjan2026release} identifies all of these by a SHA-256 list of every file (\path{RELEASE_SHA256.txt}). The original-study
tables are regenerated from the stored per-sample records by \path{manuscript/make_tables.py}, which also writes
\path{manuscript/tables/numbers.txt} and \path{manuscript/tables/MANIFEST.txt} (the 91 input files used); the Stage 1 tables, and
the filtered original tables that omit the learned-baseline rows, by \path{manuscript/make_tables_stage1.py}, which writes
\path{tables/numbers_stage1.txt} and \path{tables/MANIFEST_stage1.txt} (SHA-256 of every input file). Figures~\ref{fig:setup},
\ref{fig:avt} and~\ref{fig:errormaps} are regenerated by \path{manuscript/figures/make_figures.py} and Figure~\ref{fig:keyfig} by
\path{manuscript/figures/make_figures_stage1.py}. The command sequence is given in the release's \path{README.md}.

\clearpage%
\bibliographystyle{elsarticle-num}
\bibliography{refs}
\end{document}